\documentclass{ws-m3as}
\usepackage{color}
\newcommand{\be}{\begin{equation}}
\newcommand{\ee}{\end{equation}}
\newcommand\bes{\begin{eqnarray}} \newcommand\ees{\end{eqnarray}}
\newcommand{\bess}{\begin{eqnarray*}}
\newcommand{\eess}{\end{eqnarray*}}
\newcommand\ep{\varepsilon}

\newcommand\dd{\displaystyle}

\newcommand\R{\mathbb{R}}

\allowdisplaybreaks[4] 
\begin{document}

\markboth{J.P. Wang and M.X. Wang}{A quasilinear chemotaxis-May-Nowak model}

%
\catchline{}{}{}{}{}
%

\title{Global boundedness and finite time blow-up of solutions\\
for a quasilinear chemotaxis-May-Nowak model}

\author{Jianping Wang}

\address{School of Mathematics and Statistics, Xiamen University of Technology, Xiamen 361024, China,\\
jianping0215@163.com}

\author{Mingxin Wang\footnote{Corresponding author}}

\address{School of Mathematics and Statistics, Shanxi University, Taiyuan 030006,  China,\\
mxwang@sxu.edu.cn}

\maketitle

\begin{history}
\received{(Day Month Year)}
\revised{(Day Month Year)}
\comby{(xxxxxxxxxx)}
\end{history}

\begin{abstract}
In this paper, we introduce the nonlinear diffusion term $\nabla\cdot(D(u)\nabla u)$ into the chemotaxis-May-Nowak model to investigate the effects of $D(u)$ and chemotaxis on the global existence, boundedness, and finite time blow-up of solutions. Here, $D(u)$ generalizes the prototype $(1+u)^{m-1}$ with $m\in\R$. For the parabolic-elliptic-parabolic case, if $m>2+\frac{n}{2}-\frac2{n}$ when $n\ge3$ and $m>\frac32$ when $n=2$, then all solutions exist globally and remain bounded, whereas if $n\in\{2,3\}$ and $m<1$, finite time blow-up occurs when $\Omega$ is a ball and the initial data are radially symmetric. For the fully parabolic case, if $m>1+\frac{n}{2}-\frac2{n}$, then all solutions exist globally and remain bounded.
\end{abstract}



\section{Introduction}

The mathematical study of virus infections has attracted much attention in the last decades\cite{Be20,Be21,Be24,Bur24,Het00}. Understanding the in-host dynamics of virus is one of the central issues  in this field. Let $u(t), v(t)$ and $w(t)$ represent the densities of healthy uninfected immune cells, infected immune cells, and virus particles, respectively. The simple May-Nowak model, which describes the virus dynamics in the spatially homogeneous case, is given by (Refs. \refcite{Bo97,Nowak-B96})
 \bess\left\{\begin{array}{ll}
   u_t=-\alpha_1 u-\beta uw+\varphi,\;\;&t>0,\\[1mm]
 v_t=-\alpha_2 v+\beta uw,\;\;&t>0,\\[1mm]
 w_t=-\alpha_3 w+\gamma v,\;\;&t>0,
 \end{array}\right.
 \eess
where $\varphi$ is a source term. The basic setting of this ODE model is that the healthy uninfected cells are supplied at rate $\varphi$ and become infected on contact with virus at rate $\beta uw$ (the infected cells are supplied at rate $\beta uw$ at the same time), the virus particles are produced by the infected cells at rate $\gamma$, and the three components decay linearly with rates $\alpha_1, \alpha_2, \alpha_3$. The qualitative properties of this model had been well studied by many authors\cite{Bo97,Ko04,Nowak2000}.

The spatial movement is an indispensable factor for virus infections. To model the spatial-temporal dynamics of viral infections, Stancevic et al.\cite{Stancevic} generalized the simple May-Nowak model and proposed the following chemotaxis model:
 $$\begin{cases}
  u_t=D_1\Delta u-\chi\nabla\cdot(u\nabla v)-\alpha_1u-\beta uw+\varphi,\;\;&x\in\Omega,\;\;t>0,\\[1mm]
 \kappa v_t=D_2\Delta v-\alpha_2v+\beta uw,&x\in\Omega,\;\;t>0,\\[1mm]
 w_t=D_3\Delta w-\alpha_3w+\gamma v,&x\in\Omega,\;\;t>0,\\[1mm]
 \dd\frac{\partial u}{\partial\nu}=\frac{ \partial v}{\partial\nu}=\frac{ \partial w}{\partial\nu}=0,\;\;&x\in\partial\Omega,\ t>0,\\[1.5mm]
 u(x,0)=u_0(x),\ \kappa v(x,0)=\kappa v_0(x),\ w(x,0)=w_0(x),\ &x\in\Omega,
 \end{cases}\eqno(1.0)$$
where $\kappa=1$, $D_1,D_2$ and $D_3$ are positive constants, $\Omega\subset\R^n$ is a bounded domain with smooth boundary, the unknown functions $u(x,t), v(x,t)$ and $w(x,t)$ denote the population densities of the uninfected cells, infected cells and virus particles at location $x$ and time $t$, respectively. Compared with the simple May-Nowak model, system (1.0) is obtained by taking into account the random diffusion of the cells and virus, as well as the moving tendency of the uninfected cells toward infected cells (modelling by $-\chi\nabla\cdot(u\nabla v)$).

It was shown in Ref.~\refcite{Stancevic} that system (1.0) may exhibit Turing patterns provided that the chemotactic attraction is strong enough. Furthermore, the spontaneous emergence of singularities was also detected numerically in Ref.~\refcite{BPTM19}. Later on, Tao and Winkler provided a rigorous result on the occurrence of finite time blow-up phenomenon for the parabolic-elliptic-parabolic version ($\kappa=0$) of (1.0) in two- and three-dimensions\cite{TW-S21}. On the other hand, to exclude the blow-up phenomenon, several modifications to (1.0) have been involved. One is to weaken the chemotaxis effect: taking $\chi>0$ small enough\cite{BPTM19} or replacing $\chi$ by $\chi(1+u)^{-\delta}$ with some large $\delta$ (Refs. \refcite{XuWH-ZAMP2021,Tao-Z23,Win-19}), so that the solutions of (1.0) exist globally and remain bounded. Another one is to weaken the nonlinear infection term: replacing $uw$ by either $\frac{uw}{1+u+w}$ (Ref.~\refcite{BeT20}) or $u^\lambda w$ with $\lambda<\frac{2}{n}$ (Ref.~\refcite{Fuest19}), so that (1.0) possesses a global bounded classical solution. Moreover, the strong logistic damping effect was taken into account to prevent the aggregation phenomenon induced by chemotaxis\cite{LZ24,WangS23}.

In the study of the viral dynamics, the nonlinear diffusion mechanism, such as porous medium diffusion, can model the phenomena that the random movement of cells may be strengthened or weakened near places of large densities. Moreover, chemotaxis, the biased movement of cells along concentration gradients of a chemical signal, is also important for cells. To investigate the competing effect between the nonlinear diffusion and chemotaxis in May-Nowak model, in this paper, we consider the following chemotaxis-May-Nowak model with nonlinear diffusion:
 \bes \begin{cases}
  u_t=\nabla\cdot(D(u)\nabla u)-\chi\nabla\cdot(u\nabla v)-u-uw+\varphi,&x\in\Omega,\;\;t>0,\\
 \kappa v_t=\Delta v-v+uw,&x\in\Omega,\;\;t>0,\\
 w_t=\Delta w-w+v,&x\in\Omega,\;\;t>0,\\[0.1mm]
 \dd\frac{\partial u}{\partial\nu}=\frac{ \partial v}{\partial\nu}=\frac{ \partial w}{\partial\nu}=0,\;\;&x\in\partial\Omega,\ t>0,\\[0.1mm]
  u(x,0)=u_0(x),\ \kappa v(x,0)=\kappa v_0(x),\ w(x,0)=w_0(x),\ &x\in\Omega,
  \end{cases}\label{1.1}\ees
where $\Omega\subset\R^n$ is a bounded domain with smooth boundary; $\varphi\in C^0(\bar\Omega\times[0,\infty))$ is a given nonnegative function;   $\kappa\in\{0,1\}$ and $D(u)$ satisfies
 \[D\in C^2([0,\infty)),\;\;\;{\rm and}\;\;D(u)>0,\;\;\forall\,u\ge 0,\]
which generalizes the prototype $D(u)=(1+u)^{m-1}$ with $m\in\R$. Since we are interested in the competing effect between the nonlinear diffusion and chemotaxis, we have set other parameters such as $D_i,\alpha_i,\beta,\gamma$ equivalent to $1$ for the simplicity.

In model \eqref{1.1}, the total flux of uninfected cells is modelled as the
superposition of the diffusion flux $-D(u)\nabla u$ and the chemotactic flux $\chi u\nabla v$. One can derive this mathematical model by the macroscopic approach, as that performed in Ref.~\refcite{FLM10,Ke70,LKOH17}. We also notice that, beyond classical population dynamics models, general mathematical frameworks have been developed to model epidemics\cite{Be20,Be24,BOS22,Bur24}.

The main purpose of this paper is to investigate the effects of nonlinear diffusion and chemotaxis on the global existence, boundedness and finite time blow-up of solutions. In what follows, we always assume that $\varphi\in C^0(\bar\Omega\times[0,\infty))$ is a given nonnegative function satisfying
\bess
\|\varphi\|_{L^\infty(\Omega\times(0,\infty))}\le\varphi_*
\eess
for some $\varphi_*>0$.

We firstly provide the local existence, uniqueness and regularity of solutions of \eqref{1.1}. For the parabolic-elliptic-parabolic case (i.e., $\kappa=0$), we suppose that the initial functions $u_0$ and $w_0$ satisfy
\bes
u_0,w_0\in W^{2,\infty}(\Omega), \ u_0,w_0\ge 0,\,\not\equiv 0\;\;{\rm on}\ \bar\Omega, \;\;\mbox{and} \;\;\frac{\partial u_0}{\partial\nu}=\frac{\partial w_0}{\partial\nu}=0\;\;{\rm on}\ \partial\Omega.\label{1.3z}
\ees
And for the fully parabolic  case (i.e., $\kappa=1$), we suppose that the initial functions $u_0,v_0$ and $w_0$ satisfy
\bes\begin{cases}
u_0,v_0,w_0\in W^{2,\infty}(\Omega), \ u_0,v_0,w_0\ge 0,\,\not\equiv 0\;\;{\rm on}\ \bar\Omega, \\[1mm]
\dd\frac{\partial u_0}{\partial\nu}=\frac{\partial v_0}{\partial\nu}=\frac{\partial w_0}{\partial\nu}=0\;\;{\rm on}\ \partial\Omega.\end{cases}\label{1.5z}
\ees

\begin{lemma}\label{l2.1} Let $\kappa\in\{0,1\}$ and $n\ge2$. Assume that the initial data satisfy \eqref{1.3z} when $\kappa=0$, and \eqref{1.5z} when $\kappa=1$. Then there exist $T_{\rm max}\in(0,\infty]$ and a triple $(u,v,w)$ of functions
 \bess
 \begin{cases}
 u\in C^0(\bar\Omega\times[0,\infty))\cap C^{2,1}(\bar\Omega\times(0,T_{\rm max})),\\
w\in C^0(\bar\Omega\times[0,\infty))\cap C^{2,1}(\bar\Omega\times(0,T_{\rm max})),
 \end{cases}
 \eess
 and
  \bess
 \begin{cases}
 v\in C^{2,0}(\bar\Omega\times(0,T_{\rm max}))&\mbox{if}\;\;\kappa=0,\\
v\in C^0(\bar\Omega\times[0,\infty))\cap C^{2,1}(\bar\Omega\times(0,T_{\rm max}))
\;&\mbox{if}\;\;\kappa=1
 \end{cases}
 \eess
such that $(u,v,w)$ solves \eqref{1.1} classically in $\Omega\times(0,T_{\rm max})$ and  $u,v,w>0$ in $\bar\Omega\times(0,T_{\rm max})$. Furthermore,
\bess
if\;\; T_{\rm max}<\infty,\;\;then\;\;  \limsup_{t\to T_{\rm max}}\big(\|u(\cdot,t)\|_{L^\infty(\Omega)}+\|w(\cdot,t)\|_{L^\infty(\Omega)}\big)=\infty.
\eess
\end{lemma}

\begin{proof}
The local-in-time existence, positivity and extensibility of classical solutions can be asserted by a straightforward fixed point argument\cite{TW-S11,WinD-10}. The uniqueness of solutions can be verified by using a standard testing procedure\cite[Theorem 2.1]{WinD-10}.
\end{proof}

Then we investigate the global boundedness and finite time blow-up of solutions for the parabolic-elliptic-parabolic case (i.e., $\kappa=0$).

\begin{theorem}\label{t1.1} Let $\kappa=0$, $n\ge2$, and $m\in\R$ satisfy
\bess
 m>\begin{cases}
 2+\frac n2-\frac 2n,&n\ge3,\\
\frac32,&n=2.
 \end{cases}
 \eess
Suppose that $D(h)\ge k_Dh^{m-1}$ for all $h\ge0$ and some constant $k_D>0$. Under the condition \eqref{1.3z}, the solution $(u,v,w)$ of \eqref{1.1} exists globally and is uniformly bounded in the sense that
\bess
\mathop{\rm esssup}\limits_{t>0}\left\{\|u(\cdot,t)\|_{L^\infty(\Omega)}+\|v(\cdot,t)\|_{W^{1,\infty}(\Omega)}+\|w(\cdot,t)\|_{L^\infty(\Omega)}\right\}<\infty.
\eess
\end{theorem}

The blow-up phenomenon is detected in two and three dimensions.
When $0<m<1$ and $\Omega=B_{R}(0)$ with $R>0$, we obtain a finite-time blow-up result. In addition to \eqref{1.3z}, we moreover suppose that
 \bes
u_0 \;\,{\rm and}\,\; w_0\,\, {\rm are\, radially\, symmetric},\;\;{\rm and}\,\;w_0\;\,{\rm is\, positive\, on}\;\,\bar\Omega.
 \label{1.4z}
 \ees

\begin{theorem}\label{t1.2} Let $\kappa=0$, $n\in\{2,3\}$, $0<m<1$, and $\Omega=B_{R}(0)$ with $R>0$. Suppose that $D(h)\le K_Dh^{m-1}$ for all $h\ge0$ and some constant $K_D>0$. Then for any $\mu>0$ and $\beta\ge\alpha>0$, there exists a small $r_*=r_*(\mu,\alpha,\beta, m,n)\in(0,R)$ such that if $u_0$ and $w_0$ satisfy \eqref{1.3z}, \eqref{1.4z} and
\bes
\int_\Omega u_0{\rm d}x=\mu,\;\;\;\;\text{and}\;\;\;\alpha\le w_0\le \beta\;\;\text{on}\,\; \bar\Omega
 \label{1.5x}
 \ees
as well as
 \bes
\int_{B_{r_*}(0)} u_0{\rm d}x\ge \frac{\mu}{2},\label{1.6x}
 \ees
then the solution of \eqref{1.1} blows up in finite time.
\end{theorem}

When $m<0$, if $0<D(h)\le K_Dh^{m-1}$ for all $h\ge0$ and some constant $K_D>0$, then it is easy to infer that $0<D(h)\le \bar K_Dh^{\bar  m-1}$ for all $h\ge0$ and some constant $\bar m\in(0,1)$ and $\bar K_D>0$ since $D\in C^2([0,\infty))$. Actually, it suffices to take $\bar K_D=\max\{\sup_{h\in([0,1])}D(h),\, K_D\}$ and $\bar m$ could be any value between $(0,1)$. And hence an application of Theorem \ref{t1.3} results in a more general conclusion that involving $m<1$.

 \begin{proposition}\label{p1.2}
Let $\kappa=0$, $n\in\{2,3\}$, $m<1$, and $\Omega=B_{R}(0)$ with $R>0$. Suppose that $D(h)\le K_Dh^{m-1}$ for all $h\ge0$ and some constant $K_D>0$. Then for any $\mu>0$ and $\beta\ge\alpha>0$, there exists a small $r_*=r_*(\mu,\alpha,\beta, m,n)\in(0,R)$ such that if $u_0$ and $w_0$ satisfy \eqref{1.3z}, \eqref{1.4z}, \eqref{1.5x} and \eqref{1.6x}, then the solution of \eqref{1.1} blows up in finite time.
\end{proposition}

We recall from Ref.~\refcite{TW-S21} that, when $D(u)$ is a positive constant (corresponding to $m=1$) and $n\in\{2,3\}$, the solution of \eqref{1.1} with $\kappa=0$ will blow up in finite time for some initial data and large chemotaxis effect. In two dimensional case, Theorem \ref{t1.1} asserts the global boundedness of solutions when the diffusion exponent $m>\frac32$ and Proposition \ref{p1.2} confirms the finite time blow-up when $m<1$. However, we have to leave the gap $1<m\le \frac32$ for the future study in the case of $\kappa=0$ and $n=2$. For $n=3$ and $\kappa=0$, it is also unknown that whether the solutions blow up or not when $1<m\le 2+\frac n2-\frac 2n$.

At last, we study the global boundedness of solutions for the fully parabolic case (i.e., $\kappa=1$).

\begin{theorem}\label{t1.3}Let $\kappa=1$, $n\ge2$, $m>1+\frac n2-\frac 2n$. Suppose that $D(h)\ge k_Dh^{m-1}$ for all $h\ge0$ and some constant $k_D>0$. Under the condition \eqref{1.5z}, the solution $(u,v,w)$ of \eqref{1.1} exists globally and is uniformly bounded in the sense that
\bess
\mathop{\rm esssup}\limits_{t>0}\left\{\|u(\cdot,t)\|_{L^\infty(\Omega)}
+\|v(\cdot,t)\|_{W^{1,\infty}(\Omega)}+\|w(\cdot,t)\|_{L^\infty(\Omega)}\right\}<\infty.
\eess
\end{theorem}

We remark that, it could be a great challenge to obtain the blow-up solution for the fully parabolic case due to the lack of suitable gradient structure admitted by some Lyapunov functional.

Clearly, the prototype choice $D(u)=(1+u)^{m-1}$ satisfies the requirements of $D$ in Theorem \ref{t1.1}-Theorem \ref{t1.3} and Proposition \ref{p1.2}. Since the a priori estimates in Subsection 3.1 and Section 4 still hold even if the diffusion is degenerate (i.e., $D(0)=0$), one may construct a global weak solution which is bounded under the conditions of Theorem \ref{t1.1} and Theorem \ref{t1.3} when the diffusion is degenerate (cf. Ref. \refcite{Winkler-JDE2018}).

In order to better understand our conclusions, we recall some qualitative results concerning the classical quasilinear Keller-Segel system
\bes
 \begin{cases}
   u_t=\nabla\cdot(D(u)\nabla u)-\chi\nabla\cdot(u\nabla v),\;\;&x\in\Omega,\;\;t>0,\\
 \tau v_t=\Delta v-v+u,&x\in\Omega,\;\;t>0,
 \end{cases}\label{1.6z}
 \ees
where $\tau\in\{0,1\}$, $D(u):=(1+u)^{m-1}$. For the fully parabolic case ($\tau=1$), it had been shown that whenever $m>2-\frac2n$, all solutions exist globally and remain bounded\cite{ISY14,Tao-Wjde12}, however if $m<2-\frac2n$ and $n\ge2$ as well as $\Omega$ is a ball, then unbounded solutions were found in Ref. \refcite{Winkler10} (cf. also Refs. \refcite{C-Sjde12,C-Sjde15,Winkler-JDE2019}). In the critical case: $m=2-\frac2n$, global weak solutions were established when the total mass is small; whereas a converse result was obtained if the mass is sufficiently large\cite{L-M17,Winkler-PLMS2022}. For the parabolic-elliptic case ($\tau=0$), the conclusions are analogous to the fully parabolic case\cite{Lan20}. It is natural to ask: What is the critical exponent of \eqref{1.1} governing the occurrence of blowup? Does \eqref{1.1} and \eqref{1.6z} share the same critical exponent? We shall leave these for the future study.

The variants of equations (1.0) and \eqref{1.1} have attracted a lot of attention, for instance, immune chemotaxis models (Refs. \refcite{LKOH17,YKH20,ZSL23}) and the HIV-1-cytotoxic T lymphocytes (CTL) chemotaxis model (Ref. \refcite{WZM20}) have been studied.

The remainder of this paper is organized as follows. In Section 2, we present several preliminaries. Section 3 is devoted to proving the global boundedness and finite time blow-up for the parabolic-elliptic-parabolic case. In Section 4, we show the global boundedness of the fully parabolic version of \eqref{1.1}. For the brevity, we denote $\|\cdot\|_p:=\|\cdot\|_{L^p(\Omega)}$ for $0<p\le\infty$.

\section{Preliminaries}
 \setcounter{equation}{0} {\setlength\arraycolsep{2pt}

We begin by recalling a maximal Sobolev regularity property from Lemma 2.3 of Ref. \refcite{CT21RWA}, which plays a crucial role in establishing the boundedness results.

{\begin{lemma}\label{l2.2a} Assume that $q, r \in (1,\infty)$, $\delta\in(0,1)$ and $T>0$. Let $z$ be a classical solution of
 \bess \begin{cases}
 z_t=\Delta z-z+f,\;&x\in\Omega,\; t\in(0,T),\\[0.5mm]
\dd\frac{\partial z}{\partial\nu}=0, &x\in\partial\Omega,\; t\in(0,T),\\[0.5mm]
z(x,0)=z_0(x),\; &x\in\Omega,
 \end{cases} \eess
where $f\in C(\bar\Omega\times(0,T))$, and $z_0\in W^{2,\infty}(\Omega)$ satisfying $\frac{\partial z_0}{\partial\nu}=0$ on $\partial\Omega$. Then there exists $C=C(q, r, \delta)>0$ such that
\bess
\int_0^t {\rm e}^{-\delta r(t-s)}\|\Delta z(\cdot,s)\|_q^r{\rm d}s\le C\int_0^t {\rm e}^{-\delta r(t-s)}\|f(\cdot,s)\|_q^r{\rm d}s+C\|z_0\|_{W^{2,q}(\Omega)}^r
\eess
for all $t\in(0,T)$.
\end{lemma}}

The following fundamental lemma is well-known and will be used later.
{\begin{lemma}\label{l2.3a} Suppose that $p\ge1$, $k\in\{1,2\}$, $M>0$, and
\bess
 \begin{cases}
 q_k\in\left[1,\frac{np}{n-kp}\right)\;\;&{\rm when}\; \; p\le\frac{n}{k},\\
 q_k\in\left[1,\infty\right]\;\;&{\rm when}\; \; p>\frac{n}{k}.
 \end{cases}
 \eess
Let $T>0$ be fixed, and $z$ be a classical solution of
 \bess \begin{cases}
 z_t=\Delta z-z+f,\;&x\in\Omega,\;\;t\in(0,T),\\[0.1mm]
 \dd\frac{ \partial z}{\partial\nu}=0,\;&x\in\partial\Omega,\ t\in(0,T),\\[0.1mm]
z(x,0)=z_0(x),\ &x\in\Omega,
 \end{cases}
 \eess
where $f\in C(\Omega\times(0,T))$, and $z_0\in W^{2,\infty}(\Omega)$ satisfying $\frac{\partial z_0}{\partial\nu}=0$ on $\partial\Omega$. Then one can find $C_k=C(p,q_k,M)>0$ such that if
\bess
\|f(\cdot,t)\|_p\le M,\;\;\forall\,t\in(0,T),
\eess
then
\bess
\|\nabla z(\cdot,t)\|_{q_1}\le C_1,\;\;\;\|z(\cdot,t)\|_{q_2}\le C_2,\;\;\forall\,t\in(0,T).
\eess
\end{lemma}}
\begin{proof}
It can be proved by using the well-known smoothing properties of the Neumann heat semigroup\cite{Win-jde10}. Please refer to Lemma 3.1 and Lemma 3.2 of Ref. \refcite{FIWY16} for the details.
\end{proof}

When $\kappa=0$, we obtain the uniform-in-time $L^1$-boundedness of $u$ and $w$.
\begin{lemma}\label{l2.4z} Let $\kappa=0$, $n\ge2$. Then there exists $C>0$ such that
\bes
\int_\Omega u(x,t){\rm d}x\le C,\;\; \forall\,t\in(0,T_{\rm max}),\label{2.1y}
\ees
and
\bes
\int_\Omega w(x,t){\rm d}x\le C,\;\; \forall\,t\in(0,T_{\rm max}).\label{2.2y}
\ees
\end{lemma}

\begin{proof}
Integrating the first equation of \eqref{1.1} over $\Omega$ shows that
\bes
 \frac{\rm d}{{\rm d}t}\int_\Omega u{\rm d}x&\le& \int_\Omega \varphi{\rm d}x-\int_\Omega u{\rm d}x-\int_\Omega uw{\rm d}x\nonumber\\[1mm]
 &\le&\int_\Omega \varphi{\rm d}x-\int_\Omega u{\rm d}x,\;\; \forall\,t\in(0,T_{\rm max}),\label{2.3y}
\ees
and hence
\bess
\int_\Omega u{\rm d}x\le \max\{\|u_0\|_{L^1(\Omega)},\ \varphi_*|\Omega|\},\;\; \forall\,t\in(0,T_{\rm max}),
\eess
where $\varphi_*:=\|\varphi\|_{L^\infty(\Omega\times(0,\infty))}$, which implies \eqref{2.1y}. Integrating \eqref{2.3y} upon $(t,t+\tau)$ with $\tau:=\min\{1,\ \frac{T_{\rm max}}{2}\}$ for $t\in(0,T_{\rm max}-\tau)$ and using \eqref{2.1y} implies that
 \bess
\int_t^{t+\tau}\!\!\int_\Omega uw{\rm d}x{\rm d}s&\le&\int_\Omega u{\rm d}x+\int_t^{t+\tau}\!\!\int_\Omega\varphi{\rm d}x{\rm d}s\\[1mm]
 &\le&\int_\Omega u{\rm d}x+\varphi_*|\Omega|\le C_1,\;\; \forall\,t\in(0,T_{\rm max}-\tau)
\eess
with some $C_1>0$. Integrating the second equation of \eqref{1.1} we find that $\int_\Omega v{\rm d}x=\int_\Omega uw{\rm d}x$ and hence
\bes
\int_t^{t+\tau}\!\!\int_\Omega v{\rm d}x{\rm d}s=\int_t^{t+\tau}\!\!\int_\Omega uw{\rm d}x{\rm d}s\le C_1,\;\; \forall\, t\in(0,T_{\rm max}-\tau).\label{2.4y}
\ees
Finally we integrate the third equation of \eqref{1.1} on $\Omega$ to get
\bess
\frac{\rm d}{{\rm d}t}\int_\Omega w{\rm d}x+\int_\Omega w{\rm d}x=\int_\Omega v{\rm d}x,\;\; \forall\,t\in(0,T_{\rm max}),
\eess
which combined with \eqref{2.4y} and the ODE comparison (Lemma 2.2 of Ref. \refcite{WW-M3AS2020}) yields \eqref{2.2y}.
\end{proof}

When $\kappa=1$, the second equation of \eqref{1.1} is parabolic, we can combine it with the first equation to eliminate the nonlinearity $uw$ and then obtain uniform-in-time $L^1$-boundedness of $u$ and $v$.

\begin{lemma}\label{l2.5} Let $\kappa=1$, $n\ge2$. Then there exists $C>0$ such that
\bes
\int_\Omega u(x,t){\rm d}x\le C,\;\; \forall\,t\in(0,T_{\rm max})\label{2.5y}
\ees
and
\bes
\int_\Omega v(x,t){\rm d}x\le C,\;\; \forall\,t\in(0,T_{\rm max}).\label{2.6y}
\ees
\end{lemma}

\begin{proof}
From the first two equations in \eqref{1.1}, it is easy to infer that
\bess
\frac{\rm d}{{\rm d}t}\left(\int_\Omega u{\rm d}x+\int_\Omega v{\rm d}x\right)+\left(\int_\Omega u{\rm d}x+\int_\Omega v{\rm d}x\right)\le\int_\Omega\varphi{\rm d}x,\;\; \forall\,t\in(0,T_{\rm max}),
\eess
from which we get
\bess
\int_\Omega u{\rm d}x+\int_\Omega v{\rm d}x\le \max\{\|u_0+v_0\|_{L^1(\Omega)},\varphi_*|\Omega|\},\;\; \forall\,t\in(0,T_{\rm max}).
\eess
This provides \eqref{2.5y} and \eqref{2.6y}.
\end{proof}

Next, we shall present a conditional $L^p$ estimate for $u$ which is ensured by the Gagliardo-Nirenberg inequality and \eqref{2.1y} when $\kappa=0$ or \eqref{2.5y} when $\kappa=1$.

\begin{lemma}\label{l2.6} Let $\kappa\in\{0,\ 1\}$ and $n\ge2$. Assume that $m\ge0$, $p>1$ and $k>1$ satisfy
\bes
k<\min\left\{\frac{n(p+m-1)}{n-2},\ p+m+\frac2n-1\right\}.\label{2.7y}
\ees
Then, for any $\ep>0$, there exists $C(\ep)>0$ such that
\bess
\int_\Omega u^k{\rm d}x\le \ep\int_\Omega u^{p+m-3}|\nabla u|^2{\rm d}x+C(\ep)\;\;\;\;{\rm for\ all}\ t\in(0,T_{\rm max}).
\eess
\end{lemma}
\begin{proof}
Let
\bess
\alpha:=\frac{n(p+m-1)/2-n(p+m-1)/(2k)}{1-n/2+n(p+m-1)/2}.
\eess
In view of \eqref{2.7y}, one can derive that
\bes
\alpha\in(0,1),\;\;\; {\rm and}\;\;\frac{2k\alpha}{p+m-1}\in(0,2).\label{2.8y}
\ees
Based on \eqref{2.8y}, by using the Gagliardo-Nirenberg inequality and Young's inequality as well as \eqref{2.1y} when $\kappa=0$ (or \eqref{2.5y} when $\kappa=1$), for any $\ep>0$ one can find $C_1,C_2, C_3(\ep)>0$ such that
\bess
\int_\Omega u^k{\rm d}x&=&\big\|u^{\frac{p+m-1}{2}}\big\|_{\frac{2k}{p+m-1}}^{\frac{2k}{p+m-1}}\nonumber\\
&\le&C_1\big\|\nabla u^{\frac{p+m-1}{2}}\big\|_2^{\frac{2k\alpha}{p+m-1}}\big\|u^{\frac{p+m-1}{2}}\big\|_{\frac{2}{p+m-1}}^{\frac{2k(1-\alpha)}{p+m-1}}
+C_1\big\|u^{\frac{p+m-1}{2}}\big\|_{\frac{2}{p+m-1}}^{\frac{2k}{p+m-1}}\nonumber\\
&\le&C_2\big\|\nabla u^{\frac{p+m-1}{2}}\big\|_2^{\frac{2k\alpha}{p+m-1}}+C_2\nonumber\\
&\le&\frac{4\ep}{(p+m-1)^2}\big\|\nabla u^{\frac{p+m-1}{2}}\big\|_2^2+C_3(\ep)\nonumber\\
&=&\ep\int_\Omega u^{p+m-3}|\nabla u|^2{\rm d}x+C_3(\ep),\;\; \forall\,t\in(0,T_{\rm max}),
\eess
which gives the desired result.
\end{proof}

\section{Global boundedness versus  finite time blow-up for parabolic-elliptic-parabolic system ($\kappa=0$)}
\setcounter{equation}{0} {\setlength\arraycolsep{2pt}

\subsection{Global boundedness}

In this subsection, we always assume that the initial data satisfies \eqref{1.3z}.

We first derive the uniform-in-time $L^p$-boundedness of $u$ for any $p>1$ under the conditions of Theorem \ref{t1.1} when $n\ge3$.

\begin{lemma}\label{l3.3}
Let $\kappa=0$, $n\ge3$, and $m>2+\frac n2-\frac 2n$. Assume that $D(h)\ge k_Dh^{m-1}$ for all $h\ge0$ and some constant $k_D>0$. Then, for any given $p>1$, there exist $q>\frac{p+m+2/n-1}{m+2/n-1}$ and {$C= C(p)>0$} such that
\bess
\|u(\cdot,t)\|_p\le C,\;\;\forall\,t\in(0,T_{\rm max}),
\eess
and
 \bess
\int_0^t {\rm e}^{-\frac q2(t-s)}\|v(\cdot,s)\|_{W^{2,q}(\Omega)}^q {\rm d}s\le C,\;\;\forall\,t\in(0,T_{\rm max}).
\eess
\end{lemma}
\begin{proof} The proof shall be split into three steps.

{\it Step 1: Selecting parameters}. For any fixed $p>1$, we define $q_*:=\frac{4}{n+4}\left(p+m+\frac2n-1\right)$. Due to $m>2+\frac n2-\frac 2n$, it is easy to see that $q_*>1$ and
\bes
\frac{2+n}{m+2/n-1}<2\label{3.4a}
\ees
as well as
\bes
\frac{n+4}4<m+\frac2n-1.\label{3.4}
\ees
Then, we have from \eqref{3.4a} that
\bess
2\left(p+m+\frac2n-1\right)+n\left(m+\frac2n-1\right)>\frac{2+n}{m+2/n-1}\left(p+m+\frac2n-1\right)+n,
\eess
i.e.,
\bes
\left(p+m+\frac2n-1\right)\left(2+\frac{n(m+2/n-1)}{p+m+2/n-1}\right)
>\frac{(2+n)\left(p+m+2/n-1\right)}{m+2/n-1}+n.\quad\;\label{3.9a}
\ees
Moreover, it follows from \eqref{3.4} that
\bess
q_*>\frac{p+m+\frac2n-1}{m+2/n-1}.
\eess
Thus, due to \eqref{3.9a} one can take $q:\,\frac{p+m+2/n-1}{m+2/n-1}<q<q_*$ such that
\bes
\left(p+m+\frac2n-1\right)\left(2+\frac nq\right)>q(2+n)+n.\label{3.10a}
\ees
Taking $q':=\frac q{q-1}$, then we also have
\bes
pq'<p+m+\frac2n-1.\label{3.5}
\ees
From $q<q_*$ and \eqref{3.10a}, it is easy to see that
\bess
\max\left\{\frac{q(n+4)}{4},\ \frac{q(2+n)+n}{2+{n}/q}\right\}<p+m+\frac2n-1.
\eess
Therefore, we can fix $\sigma>1$ such that the number $\sigma':=\frac{\sigma}{\sigma-1}$ satisfies
\bes
\max\left\{\frac{q(n+4)}{4},\; \frac{q(2+n)+n}{2+{n}/q}\right\}<q\sigma'<p+m+\frac2n-1.\label{3.6}
\ees
It follows from the left inequality in \eqref{3.6} that $\sigma'>\max\left\{\frac{n+4}{4},\ \frac{2+n+{n}/q}{2+{n}/q}\right\}$, i.e.,
 \bess
\sigma<\min\left\{\frac{n+4}{n},\ \frac{2+n+{n}/q}{n}\right\}.
 \eess
Then we have
\bess
\max\left\{\frac nq-2,\ 0\right\}<\min\left\{\frac{n}{q\sigma}+2,\ 2+n+\frac nq-n\sigma\right\}.
\eess
Hence, we can fix $\rho>1$ such that
\bes
\frac nq-2<\frac{n}{\rho}<\min\left\{\frac{n}{q\sigma}+2,\ 2+n+\frac nq-n\sigma\right\}.\label{3.7}
\ees

{\it Step 2: Estimating $\int_0^t {\rm e}^{-\frac q2(t-s)}\|v(\cdot,s)\|^q_{W^{2,q}(\Omega)}{\rm d}s$}.
Making use of the third equation of \eqref{1.1} and the maximal Sobolev regularity theory in Lemma \ref{l2.2a}, one can find $C_1>0$ such that, for all $t\in(0,T_{\rm max})$ we have
\bes
\int_0^t{\rm e}^{-\frac q2(t-s)}\|\Delta w(\cdot,s)\|_\rho^q {\rm d}s\le C_1\int_0^t{\rm e}^{-\frac q2(t-s)}\|v(\cdot,s)\|_\rho^q {\rm d}s+C_1\|w_0\|_{W^{2,\rho}(\Omega)}^q.\label{3.10}
\ees
According to the left inequality in \eqref{3.7}, we can use the Sobolev embedding theorem  to derive that
\bes
C_1\|v\|_\rho^q\le C_2\|v\|_{W^{2,q}(\Omega)}^q,\;\; \forall\,t\in(0,T_{\rm max})\label{3.11}
\ees
for some $C_2>0$. Taking
\bess
\beta:=\frac{n-{n}/(q\sigma)}{2-{n}/{\rho}+n}.
\eess
It follows from the right inequality in \eqref{3.7} that
\bess
\beta\in(0,1),\;\;\;{\rm and}\;\;\sigma\beta\in(0,1).
\eess
Then, by the Gagliardo-Nirenberg inequality and \eqref{2.2y} as well as Young's inequality, there exist $C_3,C_4,C_5>0$ such that
 \bes
\int_\Omega w^{q\sigma}{\rm d}x&=&\|w\|_{q\sigma}^{q\sigma}\nonumber\\
&\le&C_3\|\Delta w\|_\rho^{q\sigma\beta}\|w\|_1^{q\sigma(1-\beta)}+C_1\|w\|_1^{q\sigma}\nonumber\\
&\le&C_4\|\Delta w\|_\rho^{q\sigma\beta}+C_4\nonumber\\
&\le&\frac{1}{2C_2}\|\Delta w\|_\rho^q+C_5,\;\; \forall\,t\in(0,T_{\rm max}).\label{3.9}
\ees
Recalling the second equation of \eqref{1.1}, in light of the classical $L^p$ theory of elliptic equations and Young's inequality, there exist $C_6,C_7>0$ such that
\bes
\|v\|_{W^{2,q}(\Omega)}^q&\le& C_6\|uw\|_q^q
\le C_7\|u\|_{q\sigma'}^{q\sigma'}+\|w\|_{q\sigma}^{q\sigma},\;\; \forall\, t\in(0,T_{\rm max}).\label{3.12}
\ees
In view of \eqref{3.12} and \eqref{3.10}-\eqref{3.9}, we have
\bess
&&\int_0^t{\rm e}^{-\frac q2(t-s)}\|v(\cdot,s)\|_{W^{2,q}(\Omega)}^q{\rm d}s\nonumber\\
&\le&\int_0^t{\rm e}^{-\frac q2(t-s)}\| w(\cdot,s)\|_{q\sigma}^{q\sigma}{\rm d}s+C_7\int_0^t{\rm e}^{-\frac q2(t-s)}\|u(\cdot,s)\|_{q\sigma'}^{q\sigma'}{\rm d}s\nonumber\\
&\le&\frac{1}{2C_2}\int_0^t{\rm e}^{-\frac q2(t-s)}\|\Delta w(\cdot,s)\|_\rho^q{\rm d}s+C_5\int_0^t{\rm e}^{-\frac q2(t-s)}{\rm d}s\nonumber\\
&&+C_7\int_0^t{\rm e}^{-\frac q2(t-s)}\|u(\cdot,s)\|_{q\sigma'}^{q\sigma'}{\rm d}s\nonumber\\
&\le&\frac{C_1}{2C_2}\int_0^t{\rm e}^{-\frac q2(t-s)}\|v(\cdot,s)\|_\rho^qds+\frac{C_1}{2C_2}\|w_0\|_{W^{2,\rho}(\Omega)}^q\\
&&+C_5\int_0^t{\rm e}^{-\frac q2(t-s)}{\rm d}s
+C_7\int_0^t{\rm e}^{-\frac q2(t-s)}\|u(\cdot,s)\|_{q\sigma'}^{q\sigma'}{\rm d}s\nonumber\\
&\le&\frac12\int_0^t{\rm e}^{-\frac q2(t-s)}\|v(\cdot,s)\|_{W^{2,q}(\Omega)}^q{\rm d}s+\frac{C_1}{2C_2}
\|w_0\|_{W^{2,\rho}(\Omega)}^q\nonumber\\
&&+C_5\int_0^t{\rm e}^{-\frac q2(t-s)}{\rm d}s+C_7\int_0^t{\rm e}^{-\frac q2(t-s)}\|u(\cdot,s)\|_{q\sigma'}^{q\sigma'}{\rm d}s\nonumber\\
&\le&\frac12\int_0^t{\rm e}^{-\frac q2(t-s)}\|v(\cdot,s)\|_{W^{2,q}(\Omega)}^q{\rm d}s+C_7\int_0^t{\rm e}^{-\frac q2(t-s)}\|u(\cdot,s)\|_{q\sigma'}^{q\sigma'}{\rm d}s\nonumber\\
&&+\frac{C_1}{2C_2}\|w_0\|_{W^{2,\rho}(\Omega)}^q+\frac{2C_5}q
\eess
for all $t\in(0,T_{\rm max})$. Thus
\bes
\int_0^t{\rm e}^{-\frac q2(t-s)}\|v(\cdot,s)\|_{W^{2,q}(\Omega)}^q{\rm d}s\le C_8\int_0^t{\rm e}^{-\frac q2(t-s)}\|u(\cdot,s)\|_{q\sigma'}^{q\sigma'}{\rm d}s+C_8\label{3.13}
\ees
for all $t\in(0,T_{\rm max})$, where
 \[C_8:=\max\left\{2C_7,\ \frac{C_1}{C_2}\|w_0\|_{W^{2,\rho}(\Omega)}^q+\frac{4C_5}q\right\}.\]

{\it Step 3: Testing the $u$-equation of \eqref{1.1}}.
Testing the first equation of \eqref{1.1} by $u^{p-1}$ and using Young's inequality one can find $C_9>0$ such that
 \bes
 &&\frac{\rm d}{{\rm d}t}\int_\Omega u^p{\rm d}x+\dd\frac q2\int_\Omega u^p{\rm d}x\nonumber\\[1mm]
&=&-p(p-1)\int_\Omega u^{p-2}D(u)|\nabla u|^2{\rm d}x+\chi p(p-1)\int_\Omega u^{p-1}\nabla u\cdot\nabla v{\rm d}x\nonumber\\[1mm]
&&-p\int_\Omega u^p{\rm d}x-p\int_\Omega u^pw{\rm d}x+p\int_\Omega \varphi u^{p-1}{\rm d}x+\frac q2\int_\Omega u^p{\rm d}x\nonumber\\[1mm]
&=&-p(p-1)\int_\Omega u^{p-2}D(u)|\nabla u|^2{\rm d}x-\chi (p-1)\int_\Omega u^p\Delta v{\rm d}x\nonumber\\[1mm]
&&-p\int_\Omega u^p{\rm d}x-p\int_\Omega u^pw{\rm d}x+p\int_\Omega \varphi u^{p-1}{\rm d}x+\frac q2\int_\Omega u^p{\rm d}x\nonumber\\[1mm]
&\le& -k_Dp(p-1)\int_\Omega u^{p+m-3}|\nabla u|^2{\rm d}x
+\chi(p-1)\int_\Omega u^p|\Delta v|{\rm d}x\nonumber\\[1mm]
&&+p\varphi_*\int_\Omega  u^{p-1}{\rm d}x+\dd\frac q2\int_\Omega u^p{\rm d}x\nonumber\\[1mm]
&\le& -k_Dp(p-1)\int_\Omega u^{p+m-3}|\nabla u|^2{\rm d}x+C_9\int_\Omega u^{pq'}{\rm d}x+\int_\Omega|\Delta v|^q{\rm d}x\nonumber\\[1mm]
&&+p\varphi_*\int_\Omega u^{p-1}{\rm d}x+\dd\frac q2\int_\Omega u^p{\rm d}x,\;\; \forall\,t\in(0,T_{\rm max}),\label{3.15}
\ees
where $\varphi_*:=\|\varphi\|_{L^\infty(\Omega\times(0,\infty))}$. It is easy to derive from \eqref{3.15} that
 \bess
&&\int_\Omega u^p{\rm d}x+k_Dp(p-1)\int_0^t{\rm e}^{-\frac q2(t-s)}\int_\Omega u^{p+m-3}|\nabla u|^2{\rm d}x{\rm d}s\nonumber\\[1mm]
&\le& \int_\Omega u_0^p{\rm d}x+C_9\int_0^t{\rm e}^{-\frac q2(t-s)}\|u\|_{pq'}^{pq'}{\rm d}s+\int_0^t{\rm e}^{-\frac q2(t-s)}\|\Delta v\|_q^q{\rm d}s\nonumber\\[1mm]
&&+p\varphi_*\int_0^t{\rm e}^{-\frac q2(t-s)}\|u\|_{p-1}^{p-1}{\rm d}s+\dd\frac q2\int_0^t{\rm e}^{-\frac q2(t-s)}\|u\|_p^p{\rm d}s,
 \;\; \forall\,t\in(0,T_{\rm max}),
  \eess
which combined with \eqref{3.13} yields that
\bes
&&\int_\Omega u^p{\rm d}x+k_Dp(p-1)\int_0^t{\rm e}^{-\frac q2(t-s)}\int_\Omega (1+u)^{p+m-3}|\nabla u|^2{\rm d}x{\rm d}s\nonumber\\[1mm]
&\le& C_9\int_0^t{\rm e}^{-\frac q2(t-s)}\|u\|_{pq'}^{pq'}{\rm d}s+C_8\int_0^t{\rm e}^{-\frac q2(t-s)}\|u\|_{q\sigma'}^{q\sigma'}{\rm d}s\nonumber\\[1mm]
&&+p\varphi_*\int_0^t{\rm e}^{-\frac q2(t-s)}\|u\|_{p-1}^{p-1}{\rm d}s+\dd\frac q2\int_0^t{\rm e}^{-\frac q2(t-s)}\|u\|_p^p{\rm d}s+\int_\Omega u_0^p{\rm d}x+C_8\qquad\label{3.16}
\ees
for all $t\in(0,T_{\rm max})$. From $m>1-\frac2n$, \eqref{3.5} and the right hand side of \eqref{3.6}, we know that
\bess
{\max\big\{pq',\,q\sigma',\,p-1,\,p\big\}}&<&p+m+\frac2n-1\\[1mm]
&=&\min\left\{\frac{n(p+m-1)}{n-2},\; p+m+\frac2n-1\right\},
\eess
and hence we employ Lemma \ref{l2.6} to obtain $C_{10}>0$ such that for all $t\in(0,T_{\rm max})$,
 \bes
&& C_9\|u\|_{pq'}^{pq'}+C_8\|u\|_{q\sigma'}^{q\sigma'}
+p\varphi_*\|u\|_{p-1}^{p-1}
+\frac q2\|u\|_p^p\nonumber\\
 &\le& \frac{k_Dp(p-1)}{2}\int_\Omega u^{p+m-3}|\nabla u|^2+C_{10}.\label{3.14}
\ees
Inserting this into \eqref{3.16} yields that for all $t\in(0,T_{\rm max})$,
\bes
 &&\int_\Omega u^p{\rm d}x+\frac{k_Dp(p-1)}{2}\int_0^t{\rm e}^{-\frac q2(t-s)}\int_\Omega u^{p+m-3}|\nabla u|^2{\rm d}x{\rm d}s\nonumber\\[1mm]
 &\le& \int_\Omega u_0^p{\rm d}x+C_8+\frac{2C_{10}}q,\label{3.17}
\ees
which provides the uniform-in-time $L^p$-boundedness of $u$. {Moreover, we have from \eqref{3.13}, \eqref{3.14} and \eqref{3.17} that
\bes
\int_0^t{\rm e}^{-\frac q2(t-s)}\|v(\cdot,s)\|_{W^{2,q}(\Omega)}^q{\rm d}s\le C_{11},\;\; \forall\, t\in(0,T_{\rm max})\label{3.18}
\ees
for some $C_{11}>0$.} Noting that
  \[\frac{p+m+2/n-1}{m+2/n-1}<q<q_*=\frac{4}{n+4}\left(p+m+\frac 2n-1\right),\]
we obtain the weighted time-space estimate for $v$ from \eqref{3.18}.
\end{proof}

From the above proof we see that the results of Lemma \ref{l3.3} still hold for $n=2$ and the corresponding condition for $m$ is $m>2$. However, this restriction $m>2$ is too strong. We shall employ a different method to get another estimate for the two dimensional case. Unfortunately, this method is only applicable for $n=2$.

\begin{lemma}\label{l3.4a} Let $\kappa=0$, $n=2$, $m>\frac32$ and $D(h)\ge k_Dh^{m-1}$ for all $h\ge0$ and some constant $k_D>0$. Then, for any
  \bes
  p>\max\left\{1,\;\frac{m}{2m-3}-m+1,\;m-2\right\},
  \label{x.1}\ees
there exist $q>\frac{p+m}{m-1}$ and $C=C(p,q)>0$ such that
 \bess
\|u(\cdot,t)\|_p+\|w(\cdot,t)\|_q\le C,\;\;\; \forall\ t\in(0,T_{\rm max}).
\eess
\end{lemma}

\begin{proof} {\it Step 1: Selecting parameters}. For the fixed $p$ satisfying \eqref{x.1}, due to $m>\frac32$ we know that
\bess
2<\frac{p+m}{m-1}<2(p+m-1)-1,
\eess
Then, for such $p$ we take
{ $$q:=\frac{\frac{p+m}{m-1}+2(p+m-1)-1}{2},\;\;\;q':=\frac{q}{q-1}.$$}
Clearly, $\frac{p+m}{m-1}<q<2(p+m-1)-1$ and
\bes
\frac{q+1}{2}<\min\{q,\ p+m-1\},\;\;\;(p+1)q'<p+m.\label{3c.18}
\ees
By the first inequality of \eqref{3c.18}, it is reasonable to choose $\theta>1$ fulfilling
\bes
\frac{q+1}{2}<\theta<\min\{q,\ p+m-1\}.\label{3.26b}
\ees
Letting $\theta':=\frac{\theta}{\theta-1}>1$, we have from \eqref{3.26b} and $2<\frac{p+m}{m-1}<q$ that
\bes
1<(q-1)\theta'<q+1.\label{3.27b}
\ees

Define
 $$f(z)=\left(z+\frac{p+m-1}{\theta}-1\right)\left(z+\frac q{\theta}-1\right)-1.$$
In view of the right hand side of \eqref{3.26b}, it is easy to see that
 \bess
f(1)=\frac{q(p+m-1)}{\theta^2}-1>0.
 \eess
Hence, we can take $\rho>1$ and close to $1$ such that $f\left(1/\rho\right)>0$, and
 \bes
 \frac{1}{\rho}<1+\frac{1}{\theta}.\label{3c.21}
\ees
Noting that $f\left(1/{\rho}\right)>0$ is equivalent to
\bess
\frac{1/\rho+q/{\theta}}{1/\rho+q/{\theta}-1}<\frac1\rho+\frac{p+m-1}{\theta},
\eess
and $\frac{1/\rho+q/{\theta}}{1/\rho+q/{\theta}-1}>1$. This enables us to fix $\delta>1$ such that
\bes
\frac{1/\rho+q/{\theta}}{1/\rho+q/{\theta}-1}<\delta<\frac1\rho+\frac{p+m-1}{\theta}.
 \label{3c.22}\ees
Setting $\delta':=\frac{\delta}{\delta-1}>1$, the left hand side of \eqref{3c.22} asserts that
\bes
\delta'<\frac1\rho+\frac q{\theta}.\label{3.31b}
\ees

{\it Step 2: Estimating $\int_\Omega u^p{\rm d}x$ and $\int_\Omega w^q{\rm d}x$.}
{We use the first two equations of \eqref{1.1} and employ Young's inequality to find $C_1>0$ such that} for any $t\in(0,T_{\rm max})$,
\bes
&&\dd\frac1p\frac{\rm d}{{\rm d}t}\int_\Omega u^p{\rm d}x+\int_\Omega u^p{\rm d}x\nonumber\\[1mm]
&=&-(p-1)\int_\Omega u^{p-2}D(u)|\nabla u|^2{\rm d}x+\chi (p-1)\int_\Omega u^{p-1}\nabla u\cdot\nabla v{\rm d}x\nonumber\\[1mm]
&&-\int_\Omega u^pw{\rm d}x+\int_\Omega \varphi u^{p-1}{\rm d}x\nonumber\\[1mm]
&=&-(p-1)\int_\Omega u^{p-2}D(u)|\nabla u|^2{\rm d}x-\dd\frac{\chi (p-1)}{p}\int_\Omega u^p\Delta v{\rm d}x\nonumber\\[1mm]
&&
-\int_\Omega u^pw{\rm d}x+\int_\Omega \varphi u^{p-1}{\rm d}x\nonumber\\[1mm]
&\le& -(p-1)\int_\Omega u^{p-2}D(u)|\nabla u|^2{\rm d}x-\dd\frac{\chi (p-1)}{p}\int_\Omega u^pv{\rm d}x\nonumber\\[1mm]
 &&+\dd\frac{\chi (p-1)}{p}\int_\Omega u^{p+1}w{\rm d}x-\int_\Omega u^pw{\rm d}x+\int_\Omega \varphi u^{p-1}{\rm d}x\nonumber\\[1mm]
&\le& -k_D(p-1)\int_\Omega u^{p+m-3}|\nabla u|^2{\rm d}x+\dd\frac{\chi (p-1)}{p}\int_\Omega u^{p+1}w{\rm d}x+\varphi_*\int_\Omega  u^{p-1}{\rm d}x\nonumber\\[1mm]
&\le& -k_D(p-1)\int_\Omega u^{p+m-3}|\nabla u|^2{\rm d}x+C_1\int_\Omega u^{(p+1)q'}{\rm d}x\nonumber\\[1mm]
&&+\dd\frac12\int_\Omega w^q{\rm d}x+\varphi_*\int_\Omega  u^{p-1}{\rm d}x.
 \label{3c.24}\ees

Testing the third equation of \eqref{1.1} by $w^{q-1}$ and using Young's inequality we have
\bes
&&\dd\frac1q\frac{\rm d}{{\rm d}t}\int_\Omega w^q{\rm d}x+\int_\Omega w^q{\rm d}x+(q-1)\int_\Omega w^{q-2}|\nabla w|^2{\rm d}x\nonumber\\[1mm]
&\le& \int_\Omega vw^{q-1}{\rm d}x\le C_2\int_\Omega v^\theta{\rm d}x+\int_\Omega w^{(q-1)\theta'}{\rm d}x\qquad\label{3c.25}
\ees
for all $t\in(0,T_{\rm max})$ and some constant $C_2>0$. {Due to \eqref{3c.21}, by the Sobolev embedding theorem and} the known $L^p$-theory of elliptic equations as well as Young's inequality, one can find $C_3,C_4,C_5>0$ fulfilling
\bess
C_2\|v\|_\theta^\theta\le C_3\|v\|_{W^{2,\rho}(\Omega)}^\theta\le C_4\|uw\|_\rho^\theta
\le C_5\|u\|_{\rho\delta}^{\theta\delta}+C_5\|w\|_{\rho\delta'}^{\theta\delta'}\;,\;\; \forall\,t\in(0,T_{\rm max}).
\eess
{Combining this with \eqref{3c.24} and \eqref{3c.25}, we obtain}
 \bes
&&\frac{\rm d}{{\rm d}t}\int_\Omega \left(\dd\frac{u^p}p+\frac{w^q}q\right){\rm d}x+\int_\Omega\left(u^p+\frac12 w^q\right){\rm d}x\nonumber\\[1mm]
&&+\int_\Omega\left(k_D(p-1)u^{p+m-3}|\nabla u|^2+(q-1)w^{q-2}|\nabla w|^2\right){\rm d}x\nonumber\\[1mm]
&\le& C_1\int_\Omega u^{(p+1)q'}{\rm d}x+\varphi_*\!\int_\Omega u^{p-1}{\rm d}x+C_5\|u\|_{\rho\delta}^{\theta\delta}
+C_5\|w\|_{\rho\delta'}^{\theta\delta'}+\!\int_\Omega w^{(q-1)\theta'}{\rm d}x\qquad\quad\label{3c.26}
\ees
for all $t\in(0,T_{\rm max})$. According to the second inequality of \eqref{3c.18} and $p-1<p+m$, we apply the two-dimensional version of Lemma \ref{l2.6} to get $C_6>0$ such that
\bes
C_1\int_\Omega u^{(p+1)q'}{\rm d}x+\varphi_*\int_\Omega  u^{p-1}{\rm d}x\le \frac{k_D(p-1)}{2}\int_\Omega u^{p+m-3}|\nabla u|^2{\rm d}x+C_6\label{3c.27}
\ees
for all $t\in(0,T_{\rm max})$. Take $\alpha:=1-\frac1{\rho\delta}\in(0,1)$. It follows from the right hand side of \eqref{3c.22} that
\bess
\frac{2\theta\delta\alpha}{p+m-1}\in(0,2).
\eess
Thus, using the Gagliardo-Nirenberg inequality, \eqref{2.1y}, and Young's inequality in sequence, one can find $C_7,C_8,C_9>0$ such that
\bes
C_5\big\|u\big\|_{\rho\delta}^{\theta\delta}&=&C_5\big\|u^{\frac{p+m-1}{2}}
\big\|_{\frac{2\rho\delta}{p+m-1}}^{\frac{2\theta\delta}{p+m-1}}\nonumber\\[1mm]
&\le& C_7\big\|\nabla u^{\frac{p+m-1}{2}}\big\|_2^{\frac{2\theta\delta\alpha}{p+m-1}}
\big\|u^{\frac{p+m-1}{2}}\big\|_{\frac2{p+m-1}}^{\frac{2\theta\delta(1-\alpha)}{p+m-1}}
+C_7\big\|u^{\frac{p+m-1}{2}}\big\|_{\frac2{p+m-1}}^{\frac{2\theta\delta}{p+m-1}}\nonumber\\[1mm]
&\le&C_8\big\|\nabla u^{\frac{p+m-1}{2}}\big\|_2^{\frac{2\theta\delta\alpha}{p+m-1}}+C_8\nonumber\\[1mm]
&\le&\dd\frac{2k_D(p-1)}{(p+m-1)^2}\big\|\nabla u^{\frac{p+m-1}{2}}\big\|_2^2+C_9\nonumber\\[1mm]
&=&\dd\frac{k_D(p-1)}{2}\int_\Omega u^{p+m-3}|\nabla u|^2{\rm d}x+C_9,\;\; \forall\,t\in(0,T_{\rm max}).\label{3c.28}
 \ees
Take $\beta:=1-\frac1{\rho\delta'}\in(0,1)$. It then follows from \eqref{3.31b} that $\frac{2\theta\delta'\beta}q\in(0,2)$. {By applying the Gagliardo-Nirenberg inequality, \eqref{2.2y} and Young's inequality in sequence, one can find $C_{10},C_{11},C_{12}>0$ such that}
\bes
C_5\|w\|_{\rho\delta'}^{\theta\delta'}&=&C_5\big\|w^{\frac q2}
\big\|_{\frac{2\rho\delta'}q}^{\frac{2\theta\delta'}q}\nonumber\\[.1mm]
&\le& C_{10}\big\|\nabla w^{\frac q2}\big\|_2^{\frac{2\theta\delta'\beta}q}
\big\|w^{\frac q2}\big\|_{\frac2q}^{\frac{2\theta\delta'(1-\beta)}q}
+C_{10}\big\|w^{\frac q2}\big\|_{\frac2q}^{\frac{2\theta\delta'}q}\nonumber\\[.1mm]
&\le&C_{11}\big\|\nabla w^{\frac q2}\big\|_2^{\frac{2\theta\delta'\beta}q}+C_{11}\nonumber\\[.1mm]
&\le&\dd\frac{2(q-1)}{q^2}\big\|\nabla w^{\frac q2}\big\|_2^2+C_{12}\nonumber\\[.1mm]
&=&\dd\frac{q-1}{2}\int_\Omega w^{q-2}|\nabla w|^2{\rm d}x+C_{12},\;\; \forall\,t\in(0,T_{\rm max}).\label{3c.29}
\ees
Similarly, we denote $\zeta:=1-\frac1{\theta'(q-1)}$, it follows from \eqref{3.27b} that
\bess
\zeta\in(0,1),\;\;\;{\rm and}\;\;\frac{2\zeta\theta'(q-1)}q\in(0,2),
\eess
which enables us to employ the Gagliardo-Nirenberg inequality and \eqref{2.2y} as well as Young's inequality to find $C_{13},C_{14},C_{15}>0$ such that
\bes
\int_\Omega w^{(q-1)\theta'}{\rm d}x&=&\big\|w^{\frac q2}\big\|_{\frac{2(q-1)\theta'}q}^{\frac{2(q-1)\theta'}q}\nonumber\\[1mm]
&\le& C_{13}\big\|\nabla w^{\frac q2}\big\|_2^{\frac{2\zeta\theta'(q-1)}q}\big\|w^{\frac q2}\big\|_{\frac2q}^{\frac{2\theta'(q-1)(1-\zeta)}q}
+C_{13}\big\|w^{\frac q2}\big\|_{\frac2q}^{\frac{2(q-1)\theta'}q}\nonumber\\[1mm]
&\le&C_{14}\big\|\nabla w^{\frac q2}\big\|_2^{\frac{2\zeta\theta'(q-1)}q}+C_{14}\nonumber\\[1mm]
&\le&\dd\frac{2(q-1)}{q^2}\big\|\nabla w^{\frac q2}\big\|_2^2+C_{15}\nonumber\\[1mm]
&=&\dd\frac{q-1}{2}\int_\Omega w^{q-2}|\nabla w|^2{\rm d}x+C_{15},\;\; \forall\,t\in(0,T_{\rm max}).\label{3c.30}
\ees
Inserting \eqref{3c.27}-\eqref{3c.30} into \eqref{3c.26} yields that
\bess
\frac{\rm d}{{\rm d}t}\left(\frac1p\int_\Omega u^p{\rm d}x+\frac1q\int_\Omega w^q{\rm d}x\right)+\int_\Omega u^p{\rm d}x+\frac12\int_\Omega w^q{\rm d}x\le C_6+C_9+C_{12}+C_{15}
\eess
for all $t\in(0,T_{\rm max})$, and hence an ODE comparison principle implies the existence of $C_{16}>0$ such that
\bess
\int_\Omega u^p{\rm d}x+\int_\Omega w^q{\rm d}x\le C_{16},\;\; \forall\,t\in(0,T_{\rm max}).
\eess
Since $q>\frac{p+m}{m-1}$, the desired conclusion follows.
\end{proof}

The above two lemmata allow us to apply a standard iterative argument of Moser type to obtain the uniform-in-time $L^\infty$-boundedness of the solution.

\begin{lemma}\label{l3.4}
Let $\kappa=0$, $n\ge2$, $m>2+\frac n2-\frac 2n$ when $n\ge3$ and $m>\frac32$ when $n=2$. Suppose that $D(h)\ge k_Dh^{m-1}$ for all $h\ge0$ and some constant $k_D>0$. Then there exists $C>0$ such that
\bess
\|u(\cdot,t)\|_\infty+\|v(\cdot,t)\|_{W^{1,\infty}(\Omega)}+\|w(\cdot,t)\|_\infty\le C,\;\; \forall\,t\in(0,T_{\rm max}).
\eess
\end{lemma}
\begin{proof} {\it Case I: $n\ge3$}. For the fixed
 $$p>\max\left\{2n,\; (2n-1)\left(m+2/n-1\right)\right\},$$
by Lemma \ref{l3.3}, there exist $q>\frac{p+m+2/n-1}{m+2/n-1}$ and $C_1>0$ such that
\bes
\|u(\cdot,t)\|_p\le C_1\;,\;\; \forall\, t\in(0,T_{\rm max})\label{3.19a}
\ees
and
\bes
\int_0^t {\rm e}^{-\frac q2(t-s)}\|v(\cdot,s)\|_q^q{\rm d}s\le C_1\;,\;\; \forall\, t\in(0,T_{\rm max}).\label{3.19}
\ees
Testing the third equation of \eqref{1.1} by $w^{q-1}$ and using Young's inequality one can find $C_2>0$ fulfilling
 \bess
\frac{\rm d}{{\rm d}t}\int_\Omega w^q{\rm d}x+\frac q2\int_\Omega w^q{\rm d}x&=&-q(q-1)\int_\Omega w^{q-2}|\nabla w|^2{\rm d}x-\frac q2\int_\Omega w^q{\rm d}x+p\int_\Omega vw^{q-1}{\rm d}x\nonumber\\[1mm]
&\le&C_2\int_\Omega v^q{\rm d}x,\;\; \forall\,t\in(0,T_{\rm max}),
\eess
which further implies that, for any $t\in(0,T_{\rm max})$,
\bess
\int_\Omega w^q{\rm d}x\le\int_\Omega w_0^q{\rm d}x+C_2\int_0^t{\rm e}^{-\frac q2(t-s)}\|v(\cdot,s)\|_q^q{\rm d}s.
\eess
This combined with \eqref{3.19} implies that
\bes
\|w(\cdot,t)\|_q\le C_3\;,\;\; \forall\,t\in(0,T_{\rm max})\label{3.20}
\ees
with some $C_3>0$. Noticing that
 $$q>\frac{p+m+2/n-1}{m+2/n-1}>2n,\;\;\;\mbox{and}\;\; p>2n.$$
Due to \eqref{3.20} and \eqref{3.19a}, for some $k>n$ there exists $C_4>0$ such that
\bess
\|(uw)(\cdot,t)\|_k\le C_4,\;\; \forall\,t\in(0,T_{\rm max}).
\eess
Recalling the second equation of \eqref{1.1}, in light of the classical $L^p$ theory of elliptic equations, there exists $C_5>0$ such that
\bess
\|v(\cdot,t)\|_{W^{2,k}(\Omega)}\le C_5,\;\; \forall\,t\in(0,T_{\rm max}),
\eess
and hence by the Sobolev embedding theory one can find $C_6>0$ fulfilling
\bes
\|v(\cdot,t)\|_{W^{1,\infty}(\Omega)}\le C_6,\;\; \forall\,t\in(0,T_{\rm max}).\label{3.21}
\ees
On basis of \eqref{3.21} and the third equation of \eqref{1.1}, we use the comparison principle to derive the uniform-in-time $L^\infty$-boundedness of $w$. Then, with the uniform-in-time $L^\infty$-boundedness of $\nabla v$ and $w$, we can perform a Moser-type iterative argument (e.g. Lemma A.1 of Ref. \refcite{Ke70}) to deduce the uniform-in-time $L^\infty$-boundedness of $u$.

{\it Case II: $n=2$}. By Lemma \ref{l3.4a}, for the fixed
 $$r>\max\left\{4,\ 3m-4,\; \frac{m}{2m-3}-m+1\right\},$$
there exist $l>\frac{r+m}{m-1}$ and $C_7>0$ such that
\bess
\|u(\cdot,t)\|_{r}+\|w(\cdot,t)\|_{l}\le C_7,\;\; \forall\,t\in(0,T_{\rm max}).
\eess
Due to $r>4$ and $l>\frac{r+m}{m-1}>4$, there exist $\theta>2$ and $C_8>0$ such that
\bess
\|(uw)(\cdot,t)\|_\theta\le C_8,\;\; \forall\,t\in(0,T_{\rm max}).
\eess
Then by the same argument as above one can prove the $W^{1,\infty}$-boundedness of $v$ and then $L^\infty$-boundedness of $w$ and $u$. The proof is completed.
\end{proof}

\noindent{\bf Proof of Theorem \ref{t1.1}.}
The proof follows from the combination of Lemma \ref{l2.1} and Lemma \ref{l3.4}.
\hfill{\fontsize{8.7pt}{8.7pt}$\Box$}

\subsection{Blow-up phenomenon}

In this subsection, we always assume that $\Omega:=B_R(0)$ and the initial data satisfy \eqref{1.3z} and \eqref{1.4z}. We shall employ the ideas from Refs. \refcite{TW-S21} and \refcite{Wang-W23} to prove Theorem \ref{t1.2}.
The following lemma is originated from Ref. \refcite{TW-S21} which enables us to tackle the nonlinear term $uw$. The restriction $n\in\{2,3\}$ is to ensure $\frac n2<\frac{n}{n-2}$.

{\begin{lemma}\label{l4.1} Let $\kappa=0$, $n\in\{2,3\}$ and $p\in\big(\frac n2,\frac{n}{n-2}\big)$, and let $\mu>0$ and $\beta\ge\alpha>0$. Then there exist $T_*=T_*(\mu,\alpha,\beta,p)>0$ and $K=K(\mu,\alpha,\beta,p)>0$ such that if $u_0$ and $w_0$ satisfy \eqref{1.3z}, \eqref{1.4z} and \eqref{1.5x} with $\mu, \alpha$ and $\beta$, then
\bes
&\dd\frac{\mu}{2}\le\int_\Omega u(x,t){\rm d}x\le 2\mu,\;\; \forall\,t\in(0,\,\min\{T_*,\, T_{\rm max}\}),&\label{3c.36}\\[1mm]
&w(x,t)\ge \dd\frac{\alpha}{2},\;\;\|w(\cdot,t)\|_\infty\le 2\beta,\;\; \forall\,x\in\Omega,\;\,t\in(0,\,\min\{T_*,\, T_{\rm max}\}),&\label{3c.37}\\[1mm]
&\|v(\cdot,t)\|_p\le K,\;\; \forall\,t\in(0,\,\min\{T_*,\, T_{\rm max}\}).&\label{3c.38}
\ees
\end{lemma}}

\begin{proof}
By following the arguments proceeded in Sec. 3 of Ref. \refcite{TW-S21}, one can find $T'=T'(\mu,\alpha,\beta,p)>0$ and $K=K(\mu,\alpha,\beta,p)>0$ such that the second inequality of \eqref{3c.36}, and \eqref{3c.37} and \eqref{3c.38} hold in $(0,\,\min\{T',\,T_{\rm max}\})$.

Making use of \eqref{3c.37}, it is easy to compute that for all $t\in(0,\min\{T',\ T_{\rm max}\})$,
\bess
\frac{\rm d}{{\rm d}t}\int_\Omega u{\rm d}x=-\int_\Omega uw{\rm d}x-\int_\Omega u{\rm d}x+\int_\Omega\varphi{\rm d}x\ge-4\mu\beta-2\mu,
\eess
which implies
\bess
\int_\Omega u(x,t){\rm d}x\ge \int_\Omega u_0{\rm d}x-2\mu(2\beta+1)t\ge \mu-2\mu(2\beta+1)t,\;\; \forall\, t\in(0,\,\min\{T',\, T_{\rm max}\}).
\eess
Let $T''=\min\big\{\frac{1}{4(2\beta+1)},\, T'\big\}$, then for all $t\in(0,\,\min\{T'',\, T_{\rm max}\})$ we find that first inequality of \eqref{3c.36} holds.
Taking $T_*=T''$, the desired conclusions are obtained.
\end{proof}

Since the initial data $(u_0,w_0)$ is radially symmetric, we know from Lemma \ref{l2.1} that $(u,v,w)$ is also radially symmetric. Therefore, in what follows we shall rewrite $(u,v,w)=(u,v,w)(r,t)$ for $r=|x|$. System \eqref{1.1} can be rewritten as
\bes
 \begin{cases}
 u_t=\dd\frac{1}{r^{n-1}}\left(r^{n-1}D(u)u_r\right)_r\\[2mm]
 \hspace{10mm}-\dd\frac{\chi}{r^{n-1}}
 \left(r^{n-1}uv_r\right)_r-u-uw+\varphi,&r\in(0,R),\;\;t\in(0,T_{\rm max}),\\[2mm]
  0=\dd\frac{1}{r^{n-1}}\left(r^{n-1}v_r\right)_r-v+uw,&r\in(0,R),\;\;t\in(0,T_{\rm max}),\\[2mm]
  w_t=\dd\frac{1}{r^{n-1}}\left(r^{n-1}w_r\right)_r-w+v,&r\in(0,R),\;\;t\in(0,T_{\rm max}),\\[1mm]
u_r=v_r=w_r=0,\;\;&r=0,\ R,\;\;t\in(0,T_{\rm max}),\\[0.1mm]
  u(r,0)=u_0,\ w(x,0)=w_0(x),\ &r\in[0,R).
 \end{cases}\qquad\label{4.8}
 \ees
Multiplying the first equation of \eqref{4.8} by $r^{n-1}$ and integrating the obtained result over $(0,r)$ yields that
 \bess
\frac{\rm d}{{\rm d}t}\int_0^r \rho^{n-1}u{\rm d}\rho
&=&\int_0^r\left(\rho^{n-1}D(u)u_\rho\right)_\rho {\rm d}\rho-\chi \int_0^r\left(\rho^{n-1}uv_\rho\right)_\rho {\rm d}\rho\nonumber\\[1mm]
&&-\int_0^r \rho^{n-1}u{\rm d}\rho-\int_0^r \rho^{n-1}uw{\rm d}\rho+\int_0^r \rho^{n-1}\varphi {\rm d}\rho\nonumber\\[1mm]
&=&r^{n-1}D(u)u_r-\chi r^{n-1}uv_r-\int_0^r \rho^{n-1}u{\rm d}\rho\nonumber\\[1mm]
&&-\int_0^r \rho^{n-1}uw{\rm d}\rho+\int_0^r \rho^{n-1}\varphi {\rm d}\rho.
 \eess
Since it follows from the second equation in \eqref{4.8} that $r^{n-1}v_r=\int_0^r \rho^{n-1}v{\rm d}\rho-\int_0^r \rho^{n-1}uw{\rm d}\rho$, we have
\bes
\frac{\rm d}{{\rm d}t}\int_0^r \rho^{n-1}u{\rm d}\rho&=&r^{n-1}D(u)u_r-\chi u\int_0^r \rho^{n-1}v{\rm d}\rho+\chi u\int_0^r \rho^{n-1}uw{\rm d}\rho\nonumber\\[1mm]
&&-\int_0^r \rho^{n-1}u{\rm d}\rho-\int_0^r \rho^{n-1}uw{\rm d}\rho+\int_0^r \rho^{n-1}\varphi {\rm d}\rho.\label{4.9}
\ees
We set from now on that
\bes
z(s,t):=n\int_0^{s^{\frac1n}}\rho^{n-1}u(\rho,t){\rm d}\rho,\;\;\; (s,t)\in\left[0,R^n\right]\times[0,T_{\rm max}).\label{4.10}
\ees
By letting $r=s^{\frac1n}$, it is easy to see that
\bess
u(r,t)=z_s(s,t),\;\;\ u_r(r,t)=ns^{\frac{n-1}{n}}z_{ss}(s,t),\;\; \forall\, (s,t)\in\left(0,R^n\right)\times[0,T_{\rm max}),
\eess
and hence we have from \eqref{4.9} that for all $(s,t)\in\left(0,R^n\right)\times(0,T_{\rm max})$,
\bes
z_t(s,t)&=&n^2s^{\frac{2n-2}{n}}D(z_s)z_{ss}-\chi nz_s\int_0^{s^{\frac1n}} \rho^{n-1}v{\rm d}\rho+\chi nz_s\int_0^{s^{\frac1n}} \rho^{n-1}uw{\rm d}\rho\nonumber\\[1mm]
&&-z-n\int_0^{s^{\frac1n}} \rho^{n-1}uw{\rm d}\rho+n\int_0^{s^{\frac1n}} \rho^{n-1}\varphi {\rm d}\rho.\label{4.11}
\ees

Inspired by Ref. \refcite{TW-S21}, the nonlocal terms in \eqref{4.11} can be transferred into a more friendly version due to Lemma \ref{l4.1}.

{\begin{lemma}\label{l4.2} Assume that $\kappa=0$, $n\in\{2,3\}$, and let $\mu>0$ and $\beta\ge\alpha>0$. Then for any $\ep\in\left(0,\frac{2}{n}\right)$, there exist $\gamma_1=\gamma_1(\mu,\alpha,\beta,\ep)>0$, $\gamma_2=\gamma_2(\alpha)>0$ and $\gamma_3=\gamma_3(\beta)>0$ as well as $T_*=T_*(\mu,\alpha,\beta,\ep)>0$ such that if $u_0$, $w_0$ satisfy \eqref{1.3z}, \eqref{1.4z} and \eqref{1.5x} with $\mu, \alpha$ and $\beta$, then the function $z$ defined in \eqref{4.10} satisfies
\bes
z_t(s,t)\ge n^2s^{\frac{2n-2}{n}}D(z_s)z_{ss}-\gamma_1\chi s^{\frac2n-\ep}z_s+\gamma_2\chi zz_s-\gamma_3z\label{4.12}
\ees
for all $(s,t)\in\left(0,R^n\right)\times(0,\,\min\{T_*,\,T_{\rm max}\})$.
\end{lemma}}
\begin{proof}
The idea of the proof comes from Lemma 4.2 of Ref. \refcite{TW-S21}. Owing to $n\in\{2,3\}$, it is easy to see that $\frac n2<\frac{n}{n-2}$. For given $\ep\in\left(0,\frac{2}{n}\right)$, we fix $p=p(\ep)\in\left(\frac n2,\frac{n}{n-2}\right)$ such that
\bess
\frac{p-1}{p}\ge \frac{2}{n}-\ep.
\eess
For such $p$, we let $T_*=T_*(\mu,\alpha,\beta,\ep)>0$ be given by Lemma \ref{l4.1}.
Due to \eqref{3c.38}, there holds
\bes
\int_0^{s^{\frac1n}} \rho^{n-1}v{\rm d}\rho&=&\dd\frac{1}{n|B_1(0)|}\int_{B_{s^{\frac1n}}(0)} v{\rm d}x\nonumber\\[1mm]
&\le&\dd\frac{1}{n|B_1(0)|}\left(\int_\Omega v^p{\rm d}x\right)^{\frac{1}{p}} \left|B_{s^{\frac1n}}(0)\right|^{\frac{p-1}{p}}\nonumber\\[1mm]
&\le&\dd\frac{K}{n|B_1(0)|} \left|B_{s^{\frac1n}}(0)\right|^{\frac{p-1}{p}}\nonumber\\[1mm]
&\le&\dd\frac{K}{n|B_1(0)|} \left|B_1(0)\right|^{\frac{p-1}{p}} s^{\frac{p-1}{p}}\nonumber\\[1mm]
&\le&\dd\frac{K}{n|B_1(0)|^{\frac1p}} \left(R^n\right)^{\frac{p-1}{p}-\left(\frac2n-\ep\right)} s^{\frac2n-\ep}
 \label{4.14}\ees
for $(s,t)\in (0,R^n)\times(0,\,\min\{T_*,\,T_{\rm max}\})$. Thus, for any $\ep\in\left(0,\frac{2}{n}\right)$, it follows from \eqref{4.14} that
\bes
-\chi nz_s\int_0^{s^{\frac1n}} \rho^{n-1}v{\rm d}\rho\ge -\gamma_1\chi s^{\frac2n-\ep}z_s\;\;\ {\rm in}\ \left(0,R^n\right)\times(0,\,\min\{T_*,\,T_{\rm max}\}),\label{4.16}
\ees
where $\gamma_1=\frac{K\left(R^n\right)^{\frac{p-1}{p}-\left(\frac2n-\ep\right)}}{|B_1(0)|^{\frac1p}}$.
It follows from \eqref{3c.37} that
\bes
\chi nz_s\int_0^{s^{\frac1n}} \rho^{n-1}uw{\rm d}\rho\ge \frac{\alpha \chi nz_s}{2}\int_0^{s^{\frac1n}} \rho^{n-1}u{\rm d}\rho=\frac{\alpha \chi zz_s}{2}\label{4.17}
\ees
for all $(s,t)\in\left(0,R^n\right)\times (0,\,\min\{T_*,\,T_{\rm max}\})$ and
 \bes
n\int_0^{s^{\frac1n}}\rho^{n-1}uw{\rm d}\rho
 &\le& 2n\beta\int_0^{s^{\frac1n}} \rho^{n-1}u{\rm d}\rho\nonumber\\[1mm]
 &=&2\beta z\;\;\ {\rm in}\ \left(0,R^n\right)\times(0,\,\min\{T_*,\,T_{\rm max}\}).\label{4.18}
\ees
Plugging \eqref{4.16}-\eqref{4.18} into \eqref{4.11} yields \eqref{4.12} with $\gamma_2=\frac\alpha2$ and $\gamma_3=2\beta+1$.
\end{proof}

We will employ the ideas from Ref. \refcite{Wang-W23} to demonstrate the blow-up phenomenon. We next turn \eqref{4.12} into {an ODI for the function} $\int_0^{r^n} s^{-\eta}z(s,t){\rm d}s$ for $r\in(0,R]$ and $\eta>0$. Here, we remark that the assumption $0<m<1$ is used in the derivations of \eqref{4.22} and \eqref{4.23}.

{\begin{lemma}\label{l4.4}
Assume that $\kappa=0$, $n\in\{2,3\}$, $0<m<1$ and $D(h)\le K_Dh^{m-1}$ for all $h\ge0$ with some $K_D>0$. Let $\mu>0$ and $\beta\ge\alpha>0$, $\ep\in\left(0,\frac2n\right)$, $0<\eta<\min\{2-\frac2n-m,\ \frac2n-\ep\}$ and $\lambda\in\left(0,2-m-\frac2n-\eta\right)$, and let $z(s,t)$ be given by \eqref{4.10} and
 \bess
y(r,t):=\int_0^{r^n} s^{-\eta}z(s,t){\rm d}s.
 \eess
Then there exist $T_*=T_*(\mu,\alpha,\beta,\ep)>0$, $\gamma_1=\gamma_1(\mu,\alpha,\beta,\ep)>0$, $\gamma_2=\gamma_2(\alpha)>0$ and $\gamma_3=\gamma_3(\beta)>0$ such that if $u_0$ and $w_0$ satisfy \eqref{1.3z}, \eqref{1.4z} and \eqref{1.5x} with $\mu, \alpha$ and $\beta$, then
 \bes
y_t(r,t)&\ge& \dd\frac{\eta\gamma_2\chi(2-\eta)}{2r^{n(2-\eta)}} y^2(r,t)-\gamma_3 y(r,t)-\gamma_1\chi r^{n(\frac2n-\ep-\eta)}z(r^n,t)\nonumber\\[1mm]
&&-\dd\frac{n^2K_D\left(2-\frac2n-\eta\right)}{m}\left(r^{\frac{n\lambda}{m}}z(r^n,t)
+\frac{r^{n\xi}}\xi\right)\label{4.22a}
\ees
for all $r\in(0,R]$ and $t\in(0,\,\min\{T_*,\,T_{\rm max}\})$, where
 \bess\xi:=\frac 1{1-m}\left(1-\frac2n-\eta-\lambda\right)+1>0.\eess
\end{lemma}}

\begin{proof} For the fixed $\ep\in\left(0,\frac2n\right)$. Let $\gamma_1,\gamma_2,\gamma_3$ and $T_*$ be given by Lemma \ref{l4.2}. It then follows from \eqref{4.12} that, as $y(r,t)=\int_0^{r^n} s^{-\eta}z(s,t){\rm d}s$,
\bes
\int_0^{r^n}\!\!s^{-\eta}z_t(s,t){\rm d}s
&\ge& \int_0^{r^n}\!\!s^{-\eta}\left(n^2s^{\frac{2n-2}{n}}D(z_s)z_{ss}-\gamma_1\chi s^{\frac2n-\ep}z_s+\gamma_2\chi zz_s-\gamma_3z\right){\rm d}s\nonumber\\[1mm]
&=&n^2\int_0^{r^n}\!\!s^{2-\frac2n-\eta}D(z_s)z_{ss}{\rm d}s-\gamma_1\chi\int_0^{r^n}\!\!s^{\frac2n-\ep-\eta}z_s{\rm d}s\nonumber\\[1mm]
&&+\gamma_2\chi\int_0^{r^n}\!\!s^{-\eta}zz_s{\rm d}s-\gamma_3y(r,t)\nonumber\\[1mm]
&=:&J_1+J_2+J_3-\gamma_3y(r,t)\label{4.21a}
\ees
for $(r,t)\in\left(0,R\right]\times(0,\,\min\{T_*,\,T_{\rm max}\})$. Since $0<m<1$ and $0<D(h)\le K_D h^{m-1}$ for all $h\ge0$, we have $D_0(\sigma):=\int_0^\sigma D(s)ds\le \frac{K_D\sigma^m}{m}$ for any $\sigma>0$. This combined with $\eta<2-\frac2n-m$ yields that for all $(r,t)\in\left(0,R\right]\times(0,\min\{T_*,T_{\rm max}\})$,
\bes
J_1&=&n^2\int_0^{r^n}s^{2-\frac2n-\eta}D(z_s)z_{ss}{\rm d}s\nonumber\\[1mm]
&=&n^2\int_0^{r^n}s^{2-\frac2n-\eta}\left(D_0(z_s)\right)_s{\rm d}s\nonumber\\[1mm]
&=&n^2s^{2-\frac2n-\eta}D_0(z_s)|_0^{r^n}-n^2\left(2-\frac2n-\eta\right)
\int_0^{r^n}s^{1-\frac2n-\eta}D_0(z_s){\rm d}s\nonumber\\[1mm]
&\ge&-n^2\left(2-\frac2n-\eta\right)\int_0^{r^n}s^{1-\frac2n-\eta}D_0(z_s){\rm d}s\nonumber\\[1mm]
&\ge&-\dd\frac{n^2K_D\left(2-\frac2n-\eta\right)}{m}\int_0^{r^n}s^{1-\frac2n-\eta}z_s^{m}{\rm d}s.\label{4.22}
\ees
Due to $\eta<2-\frac2n-m$ and $\lambda\in\left(0,2-m-\frac2n-\eta\right)$, we know that \[\frac{\lambda}{m}>0,\;\;\;{\rm and}\;\;\xi:=\frac{1}{1-m}\left(1-\frac2n-\eta-\lambda\right)+1>0.\]
Owing to $0<m<1$, we can apply Young's inequality to find that
\bes
\int_0^{r^n}s^{1-\frac2n-\eta}z_s^{m}{\rm d}s
&=&\int_0^{r^n}s^\lambda z_s^m s^{1-\frac2n-\eta-\lambda}{\rm d}s\nonumber\\[1mm]
&\le&\int_0^{r^n}s^{\frac{\lambda}{m}} z_s{\rm d}s+\int_0^{r^n}s^{\frac1{1-m}\left(1-\frac2n-\eta-\lambda\right)}{\rm d}s\nonumber\\[1mm]
&=&s^{\frac{\lambda}{m}}z|_0^{r^n}-\dd\frac{\lambda}{m}\int_0^{r^n}s^{\frac{\lambda}{m}-1}z{\rm d}s+\int_0^{r^n}s^{\frac1{1-m}\left(1-\frac2n-\eta-\lambda\right)}{\rm d}s\nonumber\\[1mm]
&\le&s^{\frac{\lambda}{m}}z|_0^{r^n}+\int_0^{r^n}s^{\frac1{1-m}\left(1-\frac2n-\eta-\lambda\right)}{\rm d}s\nonumber\\[1mm]
&\le&r^{\frac{n\lambda}{m}}z(r^n,t)
+\dd\frac{r^{n\xi}}{\xi}
\label{4.23}
\ees
for $(r,t)\in\left(0,R\right]\times(0,\min\{T_*,T_{\rm max}\})$. Plugging \eqref{4.23} into \eqref{4.22} yields that
\bes
J_1\ge-\frac{n^2K_D\left(2-\frac2n-\eta\right)}{m}\left(r^{\frac{n\lambda}{m}}z(r^n,t)
+\frac{r^{n\xi}}{\xi}\right)\label{4.24}
\ees
for $(r,t)\in\left(0,R\right]\times(0,\min\{T_*,T_{\rm max}\})$. Thanks to $\eta<\frac2n-\ep$, it is easy to see that
\bes
J_2&=&-\gamma_1\chi\int_0^{r^n}s^{\frac2n-\ep-\eta}z_s{\rm d}s\nonumber\\[1mm]
&=&-\gamma_1\chi s^{\frac2n-\ep-\eta}z|_0^{r^n}+\gamma_1\chi\left(\frac2n-\ep-\eta\right)
\int_0^{r^n}s^{\frac2n-\ep-\eta-1}z{\rm d}s\nonumber\\[1mm]
&\ge&-\gamma_1\chi r^{n(\frac2n-\ep-\eta)}z(r^n,t)\label{4.25}
\ees
for $(r,t)\in\left(0,R\right]\times(0,\min\{T_*,T_{\rm max}\})$. By H\"{o}lder's inequality,
 \bess
y(r,t)=\int_0^{r^n}s^{-\eta}z{\rm d}s&\le&\left(\int_0^{r^n}s^{1-\eta}{\rm d}s\right)^{\frac12}\left(\int_0^{r^n}s^{-\eta-1}z^2{\rm d}s\right)^{\frac12}\\[1mm]
 &\le&\left(\dd\frac{r^{n(2-\eta)}}{2-\eta}\right)^{\frac12}
 \left(\int_0^{r^n}s^{-\eta-1}z^2{\rm d}s\right)^{\frac12},
 \eess
and thus
 \bes
\int_0^{r^n}s^{-\eta-1}z^2{\rm d}s\ge \frac{2-\eta}{r^{n(2-\eta)}}y^2(r,t)\label{4.26}
 \ees
for $(r,t)\in\left(0,R\right]\times(0,\min\{T_*,T_{\rm max}\})$. In view of \eqref{4.26}, it follows that
\bes
J_3&=&\gamma_2\chi\int_0^{r^n}s^{-\eta}zz_s{\rm d}s\nonumber\\[1mm]
&=&\dd\frac{\gamma_2\chi}{2} s^{-\eta}z^2|_0^{r^n}+\frac{\eta\gamma_2\chi}{2}\int_0^{r^n}s^{-\eta-1}z^2{\rm d}s\nonumber\\[1mm]
&=&\dd\frac{\gamma_2\chi}{2} r^{-n\eta}z^2(r^n,t)+\frac{\eta\gamma_2\chi}{2}\int_0^{r^n}s^{-\eta-1}z^2{\rm d}s\nonumber\\[1mm]
&\ge&\dd\frac{\eta\gamma_2\chi(2-\eta)}{2r^{n(2-\eta)}}
y^2(r,t)\label{4.27}
\ees
for $(r,t)\in\left(0,R\right]\times(0,\,\min\{T_*,\,T_{\rm max}\})$, where we have used $\lim_{s\rightarrow0}s^{-\eta}z^2(s,t)=0$ by use of $\eta<1$ and L'H\^{o}pital's rule. Inserting \eqref{4.24}, \eqref{4.25} and \eqref{4.27} into \eqref{4.21a} we can obtain
the inequality \eqref{4.22a}.
\end{proof}

We next infer from \eqref{4.22a} that the maximal time of existence $T_{\rm max}$ must be finite.

{\begin{lemma}\label{l4.5} Assume that $\kappa=0$, $n\in\{2,3\}$, $0<m<1$ and $D(h)\le K_Dh^{m-1}$ for all $h\ge0$ and some constant $K_D>0$. Let $\mu>0$ and $\beta\ge\alpha>0$. Then there exists a small $r_*=r_*(\mu,\alpha,\beta,m,n)>0$ such that whenever $u_0$ and $w_0$ satisfy \eqref{1.3z}, \eqref{1.4z} and \eqref{1.5x} as well as
\bess
\int_{B_{r_*}(0)} u_0{\rm d}x\ge \frac{\mu}{2},
 \eess
then we have $T_{\rm max}<\infty$.
\end{lemma}}
\begin{proof} We first introduce $\ep\in\left(0,\frac2n\right)$, $0<\eta<\min\{2-\frac2n-m,\ \frac2n-\ep\}$ and $\lambda\in\left(0,2-m-\frac2n-\eta\right)$. Let $T_*=T_*(\mu,\alpha,\beta,\ep)>0$ be given by Lemma \ref{l4.2}. Suppose that $T_{\rm max}>T_*$. Then, $T_*=\min\{T_*,T_{\rm max}\}$. Let $\gamma_1=\gamma_1(\mu,\alpha,\beta,\ep)>0$, $\gamma_2=\gamma_2(\alpha)>0$ and $\gamma_3=\gamma_3(\beta)>0$ be given by Lemma \ref{l4.2}. We denote
\bess
C_1:=\frac{\mu\left(1-\left(\frac{1}2\right)^{n(1-\eta)}\right)}{2(1-\eta)|B_1(0)|}>0,\;\;
\xi:=\frac 1{1-m}\left(1-\frac2n-\eta-\lambda\right)+1>0.
\eess
Then, there exists $r_*=r_*(\mu,\alpha,\beta,m,n,\ep,\eta,\lambda)>0$ small enough such that
\bes
\frac{\eta\gamma_2(2-\eta)C_1^2}{32r_*^{n\eta}}
\ge \frac{2\mu\gamma_1 r_*^{n(\frac2n-\ep-\eta)}}{\left|B_1(0)\right|},\label{4.28}
\ees
and
\bes
\frac{\eta\gamma_2\chi(2-\eta)C_1}{16r_*^n} \ge\gamma_3,\label{4.29}
\ees
and
\bes
\frac{\eta\gamma_2(2-\eta)C_1^2}{32r_*^{n\eta}}\ge  \frac{n^2K_D\left(2-\frac2n-\eta\right)}{m}\left(\frac{2\mu r_*^{\frac{n\lambda}{m}}}{\left|B_1(0)\right|}+
\frac{r_*^{n\xi}}{\xi}\right),\quad\label{4.30}
\ees
as well as
\bes
\frac{\eta\gamma_2\chi(2-\eta)}{8r_*^{n(2-\eta)}} T_*> \frac{2}{C_1r_*^{n(1-\eta)}}.\label{4.36a}
\ees
Noticing that $T_*$ depends only on $\mu,\alpha,\beta$ and $\ep$, hence \eqref{4.36a} is possible. Suppose that
 \[\int_{B_{\frac{r_*}{2}}(0)}u_0{\rm d}x\ge\frac{\mu}{2}\]
from now on. Then we have
\bess
z(r^n,0)=\frac{1}{|B_1(0)|}\int_{B_r(0)}u_0{\rm d}x\ge\frac{1}{|B_1(0)|}\int_{B_{\frac{r_*}{2}}(0)}u_0{\rm d}x\ge \frac{\mu}{2|B_1(0)|}
\eess
for all $r\in\left(\frac{r_*}2,r_*\right)$. This implies that
\bess
y(r_*,0)&\ge&\int_{\left(\frac{r_*}{2}\right)^n}^{r_*^n} s^{-\eta}z(s,0)ds\nonumber\\[1mm]
&\ge& \dd\frac{\mu}{2|B_1(0)|}\int_{\left(\frac{r_*}{2}\right)^n}^{r_*^n} s^{-\eta}ds\nonumber\\[1mm]
&=& \dd\frac{ \mu\left(r_*^{n(1-\eta)}-\left(\frac{r_*}2\right)^{n(1-\eta)}\right)}{2(1-\eta)|B_1(0)|}\nonumber\\[1mm]
&=&C_1r_*^{n(1-\eta)}.
\eess
Hence,
\bess
\tilde T:=\sup\left\{T\in(0,T_*):\,  y(r_*,t)> \frac{C_1r_*^{n(1-\eta)}}{2},\;\;\forall\;t\in[0,T)\right\}
\eess
is well-defined, and $\tilde T\in (0,T_*]$. By the second inequality of \eqref{3c.36}, we have $z(r_*^n,t)\le z(R^n,t)\le \frac{2\mu}{\left|B_1(0)\right|}$ for all $t\in[0,T_*)$. It then follows from \eqref{4.22a} and \eqref{4.28}-\eqref{4.30} that
\bess
y_t(r_*,t)&\ge& \dd\frac{\eta\gamma_2\chi(2-\eta)}{2r_*^{n(2-\eta)}} y^2(r_*,t)-\gamma_3 y(r_*,t)-\gamma_1\chi r_*^{n(\frac2n-\ep-\eta)}z(R^n,t)\\[1mm]
&&-\dd\frac{n^2K_D\left(2-\frac2n-\eta\right)}{m}\left(r_*^{\frac{n\lambda}m}z(R^n,t)
+\frac{r_*^{n\xi}}\xi\right)\\[1mm]
&\ge& \dd\frac{\eta\gamma_2\chi(2-\eta)}{2r_*^{n(2-\eta)}} y^2(r_*,t)-\gamma_3 y(r_*,t)-\frac{2\mu\gamma_1\chi r_*^{n(\frac2n-\ep-\eta)}}{\left|B_1(0)\right|}\\[1mm]
&&-\dd\frac{n^2K_D\left(2-\frac2n-\eta\right)}{m}\left(\frac{2\mu r_*^{\frac{n\lambda}{m}}}{\left|B_1(0)\right|}+\frac{r_*^{n\xi}}\xi\right)\\[1mm]
&=& \dd\frac{\eta\gamma_2\chi(2-\eta)}{8r^{n(2-\eta)}} y^2(r_*,t)+\left(\frac{\eta\gamma_2\chi(2-\eta)}{8r_*^{n(2-\eta)}} y(r_*,t)-\gamma_3\right) y(r_*,t)\\[1mm]
&&+\dd\frac{\eta\gamma_2\chi(2-\eta)}{8r_*^{n(2-\eta)}} y^2(r_*,t)-\frac{2\mu\gamma_1\chi r_*^{n(\frac2n-\ep-\eta)}}{\left|B_1(0)\right|}+\frac{\eta\gamma_2\chi(2-\eta)}{8r_*^{n(2-\eta)}} y^2(r_*,t)\\[1mm]
&&-\dd\frac{n^2K_D\left(2-\frac2n-\eta\right)}{m}\left(\frac{2\mu  r_*^{\frac{n\lambda}{m}}}{\left|B_1(0)\right|}+\frac{r_*^{n\xi}}\xi\right)\\[1mm]
&\ge&\dd\frac{\eta\gamma_2\chi(2-\eta)}{8r_*^{n(2-\eta)}} y^2(r_*,t)+\left(\frac{\eta\gamma_2\chi(2-\eta) C_1}{16r_*^n}-\gamma_3\right) y(r_*,t)\\[1mm]
&&+\dd\frac{\eta\gamma_2\chi(2-\eta) C_1^2}{32r_*^{n\eta}}-\frac{2\mu\gamma_1\chi r_*^{n(\frac2n-\ep-\eta)}}{\left|B_1(0)\right|}\\[1mm]
&&+\frac{\eta\gamma_2\chi(2-\eta) C_1^2}{32r_*^{n\eta}}-\dd\frac{n^2K_D\left(2-\frac2n-\eta\right)}{m}\left(\frac{2\mu r_*^{\frac{n\lambda}{m}}}{\left|B_1(0)\right|}+\frac{r_*^{n\xi}}\xi\right)\\[1mm]
&\ge&\dd\frac{\eta\gamma_2\chi(2-\eta)}{8r_*^{n(2-\eta)}} y^2(r_*,t),\;\;\forall\, t\in(0,\tilde T).
 \eess
Hence, $y$ is nondecreasing in $(0,\tilde T)$. Thus, it is easy to infer that $\tilde T=T_*$ and
\bes
 y_t(r_*,t)\ge \frac{\eta\gamma_2\chi(2-\eta)}{8r_*^{n(2-\eta)}} y^2(r_*,t),\;\;\forall\, t\in(0,T_*). \label{4.37}
\ees
Integrating \eqref{4.37} over $t\in(0,T_*)$ yields that
\bess
\frac{\eta\gamma_2\chi(2-\eta)}{8r_*^{n(2-\eta)}}T_*\le \frac{1}{ y(r_*,0)}<\frac{2}{C_1r_*^{n(1-\eta)}}
\eess
This combined with \eqref{4.36a} results in a contraction. Therefore, we must have $T_{\rm max}\le T_*$.
\end{proof}

\noindent{\bf Proof of Theorem \ref{t1.2}.}
Let $r_*=r_*(\mu,\alpha,\beta,m,n)>0$ be given by Lemma \ref{l4.5}. One can obtain the conclusion by combining Lemma \ref{l4.5} and Lemma \ref{l2.1}.
\hfill{\fontsize{8.7pt}{8.7pt}$\Box$}

\section{Global boundedness in the case of $\kappa=1$}
\setcounter{equation}{0} {\setlength\arraycolsep{2pt}

Compared with the case of $\kappa=0$ (the second equation of \eqref{1.1} is a elliptic one), we have the uniform-in-time $L^1$-estimate for $v$ and a better regularity property for $w$, which will help us establish the global boundedness of the solution under a more relaxed condition on the diffusion exponent $m$. We always assume that the initial data satisfies \eqref{1.5z} in this section. On the basis of the $L^1$ boundedness feature in \eqref{2.6y}, one can infer the following regularity of $w$ from Lemma \ref{l2.3a}.

\begin{proposition}\label{p2.4} Let $\kappa=1$ and $n\ge2$. The, for any $p\in\big[1,\frac{n}{n-2}\big)$, there exists $C(p)>0$ fulfilling
\bess
\|w(\cdot,t)\|_p\le C(p),\;\;\forall\, t\in(0,T_{\rm max}).
\eess
\end{proposition}

In the case of $n\ge3$, we shall apply the maximal Sobolev inequality for heat equations to obtain the uniform-in-time $L^p$-estimate for $u$. Although the idea used in the following lemma is parallel to Lemma \ref{l3.3}, the selection of the parameters is different and more complicate. Therefore, we will give the necessary details in the proof.

\begin{lemma}\label{l5.2}
Let $\kappa=1$, $n\ge3$ and $m>1+\frac n2-\frac 2n$. Suppose that $D(h)\ge k_Dh^{m-1}$ for all $h\ge0$ and some constant $k_D>0$. Then, for any $p>\max\left\{1,\ \frac{2m}{n-2}-\frac2n\right\}$, there exist $q>\frac{p+m+2/n-1}{m+2/n-1}$ and {$C=C(p)>0$} such that
\bess
\|u(\cdot,t)\|_p\le C,\;\;\forall\,t\in(0,T_{\rm max}),
\eess
and
\bess
\int_0^t{\rm e}^{-\frac q2(t-s)}\|v(\cdot,s)\|_{W^{2,q}(\Omega)}^q{\rm d}s\le C,\;\; \forall\,t\in(0,T_{\rm max}).
\eess
\end{lemma}
\begin{proof}

{\it Step 1: Choosing parameters}. For any fixed $p>\max\left\{1,\ \frac{2m}{n-2}-\frac2n\right\}$, due to $m>1+\frac{n}2-\frac2n=\frac{n}{2\left(\frac{n}{n-2}\right)}+2-\frac2n$, we take $p_*\in\big(1,\frac{n}{n-2}\big)$
such that
\bess
m>\frac{n}{2p_*}+2-\frac2n,\;\;\;{\rm i.e.}, \;\;\frac{2+{n}/{p_*}}{m+2/n-1}<2.
\eess
Let $q_*:=\frac{4}{{n}/{p_*}+4}\left(p+m+\frac2n-1\right)$. Then, we have
\bes
q_*>\frac{p+m+2/n-1}{m+2/n-1}\label{5.3b}
\ees
and
\bess
2\left(p+m+\frac2n-1\right)+n\left(m+\frac2n-1\right)>\frac{2+{n}/{p_*}}{m+2/n-1}\left(p+m+\frac2n-1\right)+n,
\eess
i.e.,
\bes
&&\left(p+m+\frac2n-1\right)\left(2+\frac{n(m+2/n-1)}{p+m+2/n-1}\right)\nonumber\\[1mm]
&>&\frac{\left(2+{n}/{p_*}\right)\left(p+m+2/n-1\right)}{m+2/n-1}+n.\label{5.4b}
\ees
Thus, by \eqref{5.3b} and \eqref{5.4b} we can fix $\frac{p+m+2/n-1}{m+2/n-1}<q<q_*$ such that
\bess
\left(p+m+\frac2n-1\right)\left(2+\frac nq\right)>q\left(2+\frac{n}{p_*}\right)+n
\eess
i.e.,
\bes
\frac{q\left(2+{n}/{p_*}\right)+n}{2+{n}/q}<p+m+\frac2n-1.\label{5.6b}
\ees
On the other hand, on basis of $p>\frac{2m}{n-2}-\frac2n$ and $p_*\in\big[1,\frac{n}{n-2}\big)$ as well as $m>1$, we infer that $p_*<\frac{p+m+2/n-1}{m+2/n-1}$ and hence
\bes
p_*<q.\label{5.4}
\ees
We let $q':=\frac q{q-1}$, then it follows from $\frac{p+m+2/n-1}{m+2/n-1}<q$ that
\bes
pq'<p+m+\frac2n-1,\label{5.5}
\ees
and by $q<q_*$ we also have
\bess
\frac{q\left({n}/{p_*}+4\right)}{4}<p+m+\frac2n-1.
\eess
Therefore, this combined with \eqref{5.6b} enables us to fix $\theta>1$ and hence $\theta':=\frac{\theta}{\theta-1}$ fulfilling
\bes
\max\left\{\frac{q\left(2+{n}/{p_*}\right)+n}{2+{n}/q},\ \frac{q\left({n}/{p_*}+4\right)}{4}\right\}<q\theta'<p+m+\frac2n-1.\label{5.6}
\ees
It follows from the left inequality in \eqref{5.6} that
 \bess
 \theta<\min\left\{\frac{p_*}{n}\left(2+\frac nq\right)+1,\;  \frac{n/{p_*}+4}{{n}/{p_*}}\right\},\eess
then, we have
\bess
\max\left\{0,\ \frac nq-2\right\}<2+\frac nq+\frac{n}{p_*}-\frac{n\theta}{p_*}.
\eess
Moreover, due to the left inequality in \eqref{5.6} and $p_*\in\left(1,\frac{n}{n-2}\right)$, it is easy to see that $\theta'>\frac{{n}/{p_*}+4}{4}>\frac{n+2}{4}$ and so
$\frac nq-2<\frac{n}{q\theta}+2$. Therefore,
 \bess
\frac nq-2<\min\left\{\frac{n}{q\theta}+2,\ 2+\frac nq+\frac{n}{p_*}-\frac{n\theta}{p_*}\right\}.
 \eess
{Since $\frac nq-2<n$, we also have
\bess
\frac nq-2<\min\left\{n,\ \frac{n}{q\theta}+2,\ 2+\frac nq+\frac{n}{p_*}-\frac{n\theta}{p_*}\right\}=:\varrho.
\eess
Taking $0<\ep\ll1$ such that $0<\varrho-\ep<n$ and
\bess
\frac nq-2<\varrho-\ep.
\eess
Define $\rho:=\frac{n}{\varrho-\ep}$, then we have $\rho>1$ and $\frac{n}{\rho}=\varrho-\ep$ satisfies
\bes
\frac nq-2<\frac{n}{\rho}<\min\left\{\frac{n}{q\theta}+2,\; 2+\frac nq+\frac{n}{p_*}-\frac{n\theta}{p_*}\right\}.\label{5.7}
\ees}

{\it Step 2: Estimating $\int_0^t{\rm e}^{-\frac q2(t-s)}\|v(\cdot,s)\|_{W^{2,q}(\Omega)}^q{\rm d}s$}. Taking
\bess
\beta:=\frac{{n}/{p_*}-{n}/{(q\theta)}}{2-{n}/{\rho}+{n}/{p_*}}.
\eess
Due to \eqref{5.4}, it is easy to see that
\bes
\frac1{q\theta}<\frac1q<\frac1{p_*}.\label{5.7a}
\ees
It follows from \eqref{5.7a} and the right inequality in \eqref{5.7} that
\bess
\beta\in(0,1)\ \ {\rm and}\ \ \theta\beta\in(0,1).
\eess
Then, we use the Gagliardo-Nirenberg inequality, {elliptic regularity theory} and Proposition \ref{p2.4} with $p_*\in\big(1,\frac{n}{n-2}\big)$ as well as Young's inequality to find $C_1,C_2,C_3>0$ fulfilling
\bes
\int_\Omega w^{q\theta}{\rm d}x&=&\|w\|_{q\theta}^{q\theta}\nonumber\\[1mm]
&\le&C_1\|\Delta w\|_\rho^{q\theta\beta}\|w\|_{p_*}^{q\theta(1-\beta)}+C_1\|w\|_{p_*}^{q\theta}\nonumber\\[1mm]
&\le&C_2\|\Delta w\|_\rho^{q\theta\beta}+C_2\nonumber\\[1mm]
&\le&\|\Delta w\|_\rho^q+C_3,\;\; \forall\,t\in(0,T_{\rm max}).\label{5.9}
\ees
Recalling the third equation of \eqref{1.1}, we employ the maximal Sobolev regularity theory in Lemma \ref{l2.2a} to get $C_4>0$ such that for all $t\in(0,T_{\rm max})$,
\bes
\int_0^t{\rm e}^{-\frac q2(t-s)}\|\Delta w(\cdot,s)\|_\rho^q{\rm d}s\le C_4\int_0^t{\rm e}^{-\frac q2(t-s)}\|v(\cdot,s)\|_\rho^q{\rm d}s+C_4\|w_0\|_{W^{2,\rho}(\Omega)}^q.\qquad\label{5.10}
\ees
Because of \eqref{2.6y} and the left inequality in \eqref{5.7}, we may use Ehrling's inequality to derive that, there exists $C_5>0$ such that
\bes
C_4\|v\|_\rho^q\le \frac12\|v\|_{W^{2,q}(\Omega)}^q+C_5,\;\; \forall\,t\in(0,T_{\rm max}).\label{5.11}
\ees
By using the maximal Sobolev regularity theory in Lemma \ref{l2.2a} and Young's inequality, we infer from the second equation of \eqref{1.1} that,  there exist $C_6,C_7>0$ such that for all $t\in(0,T_{\rm max})$,
 \bess
&&\int_0^t{\rm e}^{-\frac q2(t-s)}\|v(\cdot,s)\|_{W^{2,q}(\Omega)}^q{\rm d}s\nonumber\\[1mm]
&\le& C_6\int_0^t{\rm e}^{-\frac q2(t-s)}\|(uw)(\cdot,s)\|_q^q{\rm d}s+C_6\|v_0\|_{W^{2,q}(\Omega)}^q\nonumber\\[1mm]
&\le&C_7\int_0^t{\rm e}^{-\frac q2(t-s)}\|u(\cdot,s)\|_{q\theta'}^{q\theta'}{\rm d}s+\int_0^t{\rm e}^{-\frac q2(t-s)}\|w(\cdot,s)\|_{q\theta}^{q\theta}{\rm d}s
+C_6\|v_0\|_{W^{2,q}(\Omega)}^q.
\eess
This combined with \eqref{5.9}-\eqref{5.11} allows us to derive
\bess
&&\int_0^t{\rm e}^{-\frac q2(t-s)}\|v(\cdot,s)\|_{W^{2,q}(\Omega)}^q{\rm d}s\\[1mm]
&\le&\int_0^t\!{\rm e}^{-\frac q2(t-s)}\| w(\cdot,s)\|_{q\theta}^{q\theta}{\rm d}s+C_7\int_0^t\!{\rm e}^{-\frac q2(t-s)}\|u(\cdot,s)\|_{q\theta'}^{q\theta'}{\rm d}s+C_6\|v_0\|_{W^{2,q}(\Omega)}^q\\[1mm]
&\le&\int_0^t{\rm e}^{-\frac q2(t-s)}\|\Delta w(\cdot,s)\|_\rho^q{\rm d}s+C_3\int_0^t{\rm e}^{-\frac q2(t-s)}{\rm d}s\\[1mm]
&&+C_7\int_0^t{\rm e}^{-\frac q2(t-s)}\|u(\cdot,s)\|_{q\theta'}^{q\theta'}
+C_6\|v_0\|_{W^{2,q}(\Omega)}^q\\[1mm]
&\le&C_4\int_0^t{\rm e}^{-\frac q2(t-s)}\|v(\cdot,s)\|_\rho^qds+C_4\|w_0\|_{W^{2,\rho}(\Omega)}^q
+C_3\int_0^t{\rm e}^{-\frac q2(t-s)}{\rm d}s\\[1mm]
&&+C_7\int_0^t{\rm e}^{-\frac q2(t-s)}\|u(\cdot,s)\|_{q\theta'}^{q\theta'}{\rm d}s+C_6\|v_0\|_{W^{2,q}(\Omega)}^q\\[1mm]
&\le&\dd\frac12\int_0^t{\rm e}^{-\frac q2(t-s)}\|v(\cdot,s)\|_{W^{2,q}(\Omega)}^q{\rm d}s
+C_4\|w_0\|_{W^{2,\rho}(\Omega)}^q\\[1mm]
&&+C_7\!\int_0^t\!{\rm e}^{-\frac q2(t-s)}\|u(\cdot,s)\|_{q\theta'}^{q\theta'}{\rm d}s\!+\!(C_5\!+\!C_3)\!\int_0^t\!{\rm e}^{-\frac q2(t-s)}{\rm d}s\!+C_6\|v_0\|_{W^{2,q}(\Omega)}^q\\[1mm]
&\le&\dd\frac12\int_0^t{\rm e}^{-\frac q2(t-s)}\|v(\cdot,s)\|_{W^{2,q}(\Omega)}^q
+C_4\|w_0\|_{W^{2,\rho}(\Omega)}^q\\[1mm]
&&+C_7\int_0^t{\rm e}^{-\frac q2(t-s)}\|u(\cdot,s)\|_{q\theta'}^{q\theta'}{\rm d}s
+\dd\frac{2(C_5+C_3)}q+C_6\|v_0\|_{W^{2,q}(\Omega)}^q,
\eess
and thus
\bes
\int_0^t{\rm e}^{-\frac q2(t-s)}\|v(\cdot,s)\|_{W^{2,q}(\Omega)}^q{\rm d}s
\le C_8\int_0^t{\rm e}^{-\frac q2(t-s)}\|u(\cdot,s)\|_{q\theta'}^{q\theta'}{\rm d}s+C_8\label{5.13}
\ees
for all $t\in(0,T_{\rm max})$, where \bess
C_8=2C_7+2C_4\|w_0\|_{W^{2,\rho}(\Omega)}^q+2C_6\|v_0\|_{W^{2,q}(\Omega)}^q
+\frac{4(C_5+C_3)}q.\eess

{\it Step 3: Testing the $u$-equation of \eqref{1.1}}.
By the same derivation of \eqref{3.16}, we test the first equation of \eqref{1.1} by $u^{p-1}$ and apply \eqref{5.13} and Young's inequality to get $C_9>0$ such that
\bes
&&\int_\Omega u^p{\rm d}x+k_Dp(p-1)\int_0^t{\rm e}^{-\frac q2(t-s)}
\int_\Omega u^{p+m-3}|\nabla u|^2{\rm d}x{\rm d}s\nonumber\\[1mm]
&\le& C_9\int_0^t{\rm e}^{-\frac q2(t-s)}\|u\|_{pq'}^{pq'}{\rm d}s+C_8\int_0^t{\rm e}^{-\frac q2(t-s)}\|u\|_{q\theta'}^{q\theta'}{\rm d}s
+p\varphi_*\int_0^t{\rm e}^{-\frac q2(t-s)}\|u\|_{p-1}^{p-1}{\rm d}s\nonumber\\[1mm]
&&+\dd\frac q2\int_0^t{\rm e}^{-\frac q2(t-s)}\|u\|_p^p{\rm d}s+\int_\Omega u_0^p{\rm d}x+C_8,\;\; \forall\, t\in(0,T_{\rm max}).\label{5.16}
\ees
From \eqref{5.5}, the right hand side of \eqref{5.6} and $m>1-\frac2n$, we have
\bess
{\max\big\{pq',q\theta',p-1,p\big\}}<p+m+\frac2n-1=\min\left\{\frac{n(p+m-1)}{n-2},\ p+m+\frac2n-1\right\},
\eess
which enable us to apply Lemma \ref{l2.6} to find $C_{10}>0$ such that for all $t\in(0,T_{\rm max})$,
\bes
 &&C_9\|u\|_{pq'}^{pq'}+C_8\|u\|_{q\theta'}^{q\theta'}+p\varphi_*\|u\|_{p-1}^{p-1}
+\frac q2\|u\|_p^p\nonumber\\[1mm]
&\le& \frac{k_Dp(p-1)}{2}\int_\Omega u^{p+m-3}|\nabla u|^2{\rm d}x+C_{10}.\label{5.18a}
\ees
This combined with \eqref{5.16} infers that for any $t\in(0,T_{\rm max})$,
 \bes
 &&\int_\Omega u^p{\rm d}x+\dd\frac{k_Dp(p-1)}{2}\int_0^t{\rm e}^{-\frac q2(t-s)}\int_\Omega u^{p+m-3}|\nabla u|^2{\rm d}x {\rm d}s\nonumber\\[1mm]
 &\le& \int_\Omega u_0^p{\rm d}x+C_8+\frac{2C_{10}}q.\label{5.19a}
\ees
This provides the uniform-in-time $L^p$-boundedness of $u$. We then combine \eqref{5.19a} and \eqref{5.13} with \eqref{5.18a} to get $C_{11}>0$ such that
\bes
\int_0^t{\rm e}^{-\frac q2(t-s)}\|v(\cdot,s)\|_{W^{2,q}(\Omega)}^q{\rm d}s\le C_{11},\;\; \forall\, t\in(0,T_{\rm max}).\label{5.17a}
\ees
Since $\frac{p+m+\frac2n-1}{m+\frac2n-1}<q<q_*$, we obtain the weighted time-space estimate for $v$ from \eqref{5.17a}.
\end{proof}

In the case of $n=2$, we have from Proposition \ref{p2.4} that, for any $p>1$ there exists $C_1=C_1(p)>0$ such that $\|w(\cdot,t)\|_p<C_1$ for all $t\in(0,T_{\rm max})$. This will help us build some powerful estimations. We notice that the assumption for $m$ in Theorem \ref{t1.3} with $n=2$ is actually $m>1$.

\begin{lemma}\label{l5.3}
Let $\kappa=1$, $n=2$ and $m>1$. Suppose that $D(h)\ge k_Dh^{m-1}$ for all $h\ge0$ and some constant $k_D>0$. Then, for any $p>1$, there exists $C=C(p)>0$ such that
\bess
\|u(\cdot,t)\|_p\le C,\;\;\forall\,t\in(0,T_{\rm max}).
\eess
\end{lemma}
\begin{proof}
For any fixed $p>1$, due to $m>1$ we can pick $q\in\left(\frac{p+m}{m},\ p+m\right)$ such that $q':=\frac q{q-1}<\frac{p+m}{p}$, i.e.,
\bes
pq'<p+m.\label{5.18}
\ees
For such $q$, we take $\theta>1$ and $\theta':=\frac{\theta}{\theta-1}$ fulfilling
\bes
q<q\theta<p+m.\label{5.19}
\ees
Similar to the derivation of \eqref{3.15}, we test the first equation of \eqref{1.1} by $u^{p-1}$ and use Young's inequality to get $C_1>0$ such that
 \bess
 \frac{\rm d}{{\rm d}t}\int_\Omega u^p{\rm d}x+\dd\frac q2\int_\Omega u^p{\rm d}x
&\le& -k_Dp(p-1)\int_\Omega u^{p+m-3}|\nabla u|^2{\rm d}x+C_1\int_\Omega u^{pq'}{\rm d}x\nonumber\\[1mm]
&&+\int_\Omega|\Delta v|^q{\rm d}x+p\varphi_*\int_\Omega  u^{p-1}{\rm d}x+\dd\frac q2\int_\Omega u^p{\rm d}x
\eess
for all $t\in(0,T_{\rm max})$, where $\varphi_*:=\|\varphi\|_{L^\infty(\Omega\times(0,\infty))}$, and hence
\bes
&&\int_\Omega u^p{\rm d}x+k_Dp(p-1)\int_0^t{\rm e}^{-\frac q2(t-s)}\int_\Omega u^{p+m-3}|\nabla u|^2{\rm d}x{\rm d}s\nonumber\\[1mm]
&\le& \int_\Omega u_0^p{\rm d}x+C_1\int_0^t{\rm e}^{-\frac q2(t-s)}\|u(\cdot,s)\|_{pq'}^{pq'}{\rm d}s
+\int_0^t{\rm e}^{-\frac q2(t-s)}\|\Delta v(\cdot,s)\|_q^q{\rm d}s\nonumber\\[1mm]
&&+p\varphi_*\int_0^t{\rm e}^{-\frac q2(t-s)}\|u(\cdot,s)\|_{p-1}^{p-1}{\rm d}s+\dd\frac q2\int_0^t{\rm e}^{-\frac q2(t-s)}\|u(\cdot,s)\|_p^p{\rm d}s\label{5.20}
\ees
for all $t\in(0,T_{\rm max})$.

By the second equation of \eqref{1.1} and the maximal Sobolev regularity theory in Lemma \ref{l2.2a} as well as Young's inequality and Proposition \ref{p2.4}, one can find $C_2,C_3,C_4>0$ such that
 \bess
&&\int_0^t {\rm e}^{-\frac q2(t-s)}\|\Delta v(\cdot,s)\|_q^q{\rm d}s\\
&\le& C_2\int_0^t {\rm e}^{-\frac q2(t-s)}\|(uw)(\cdot,s)\|_q^q{\rm d}s+C_2\|v_0\|_{W^{2,q}(\Omega)}^q\\
&\le&C_3\!\int_0^t\!{\rm e}^{-\frac q2(t-s)}\|u(\cdot,s)\|_{q\theta}^{q\theta}{\rm d}s
+C_3\!\int_0^t\!{\rm e}^{-\frac q2(t-s)}\|w(\cdot,s)\|_{q\theta'}^{q\theta'}{\rm d}s+C_2\|v_0\|_{W^{2,q}(\Omega)}^q\\
&\le&C_3\int_0^t {\rm e}^{-\frac q2(t-s)}\|u(\cdot,s)\|_{q\theta}^{q\theta}{\rm d}s
+C_4\int_0^t {\rm e}^{-\frac q2(t-s)}{\rm d}s+C_2\|v_0\|_{W^{2,q}(\Omega)}^q\\
&\le&C_3\int_0^t {\rm e}^{-\frac q2(t-s)}\|u(\cdot,s)\|_{q\theta}^{q\theta}{\rm d}s+\dd\frac{2C_4}q+C_2\|v_0\|_{W^{2,q}(\Omega)}^q,\;\; \forall\,t\in(0,T_{\rm max}).
\eess
Inserting this into \eqref{5.20} yields that for any $t\in(0,T_{\rm max})$ we have
\bes
&&\int_\Omega u^p{\rm d}x+k_Dp(p-1)\int_0^t{\rm e}^{-\frac q2(t-s)}\int_\Omega u^{p+m-3}|\nabla u|^2{\rm d}x{\rm d}s\nonumber\\
&\le& \int_\Omega u_0^p{\rm d}x+C_1\int_0^t{\rm e}^{-\frac q2(t-s)}\|u(\cdot,s)\|_{pq'}^{pq'}{\rm d}s+C_3\int_0^t {\rm e}^{-\frac q2(t-s)}\|u(\cdot,s)\|_{q\theta}^{q\theta}{\rm d}s\nonumber\\
&&+p\varphi_*\int_0^t{\rm e}^{-\frac q2(t-s)}\|u(\cdot,s)\|_{p-1}^{p-1}{\rm d}s+\dd\frac q2\int_0^t{\rm e}^{-\frac q2(t-s)}\|u(\cdot,s)\|_p^p{\rm d}s\nonumber\\
 &&+\frac{2C_4}q+C_2\|v_0\|_{W^{2,q}(\Omega)}^q.\qquad\label{5.21}
\ees
Recalling Lemma \ref{l2.6} with $n=2$, on the basis of \eqref{5.18} and \eqref{5.19} as well as $p<p+m$, there exists $C_5>0$ such that
 \bess
 && C_1\|u\|_{pq'}^{pq'}+C_3\|u\|_{q\theta}^{q\theta}+p\varphi_*\|u\|_{p-1}^{p-1}
+\frac q2\|u\|_p^p\\[0.1mm]
&\le& p(p-1)k_D\int_\Omega u^{p+m-3}|\nabla u|^2{\rm d}x+C_5,\;\; \forall\,t\in(0,T_{\rm max}).
\eess
This in conjunction with \eqref{5.21} implies that
 \bess
\int_\Omega u^p{\rm d}x\le \int_\Omega u_0^p{\rm d}x+\frac{2C_4}q+C_2\|v_0\|_{W^{2,q}(\Omega)}^q+\frac{2C_5}q,\;\; \forall\,t\in(0,T_{\rm max}),
\eess
which asserts the uniform-in-time $L^p$-boundedness of $u$.
\end{proof}

The above two lemmata allow us to infer the uniform-in-time $L^\infty$-boundedness of the solution.

\begin{lemma}\label{l5.4}
Let $\kappa=1$, $n\ge2$ and $m>1+\frac n2-\frac 2n$. Suppose that $D(h)\ge k_Dh^{m-1}$ for all $h\ge0$ and some constant $k_D>0$. Then there exists $C>0$ such that
\bess
\|u(\cdot,t)\|_\infty+\|v(\cdot,t)\|_{W^{1,\infty}(\Omega)}+\|w(\cdot,t)\|_\infty\le C,\;\; \forall\,t\in(0,T_{\rm max}).
\eess
\end{lemma}

\begin{proof}
We shall claim the uniform-in-time $W^{1,\infty}(\Omega)$-boundedness of $v$ firstly.

{\it Case I: $n\ge3$}. For fixed $p>\max\left\{2n,\ (2n-1)\left(m+\frac2n-1\right),\ \frac{2m}{n-2}-\frac2n\right\}$, from Lemma \ref{l5.2} there exist $q>\frac{p+m+2/n-1}{m+2/n-1}$ and $C_1>0$ such that
\bes
\|u(\cdot,t)\|_p\le C_1\;,\;\; \forall\, t\in(0,T_{\rm max})\label{5.19b}
\ees
and
\bes
\int_0^t {\rm e}^{-\frac q2(t-s)}\|v(\cdot,s)\|_q^q{\rm d}s\le C_1\;,\;\; \forall\, t\in(0,T_{\rm max}).\label{5.19c}
\ees
By the same reason of \eqref{3.20}, we derive from the $w$-equation of \eqref{1.1} and \eqref{5.19c} that, there is $C_2>0$ fulfilling
\bess
\|w(\cdot,t)\|_q\le C_2\;,\;\; \forall\,t\in(0,T_{\rm max}),
\eess
which combined with \eqref{5.19b} and $p>2n$ and $q>\frac{p+m+2/n-1}{m+2/n-1}>2n$ implies that, for some $\theta>n$ there exists $C_3>0$ such that
\bess
\|(uw)(\cdot,t)\|_\theta\le C_3\;,\;\; \forall\,t\in(0,T_{\rm max}).
\eess
In light of Lemma \ref{l2.3a}, we then derive from the $v$-equation of \eqref{1.1} that, there exists $C_4>0$ such that
\bess
\|v(\cdot,t)\|_{W^{1,\infty}(\Omega)}\le C_4,\;\; \forall\,t\in(0,T_{\rm max}).
\eess

{\it Case II: $n=2$}. By Proposition \ref{p2.4} and Lemma \ref{l5.3}, there exists $C_5>0$ such that
\bess
\|u(\cdot,t)w(\cdot,t)\|_6\le C_5,\;\; \forall\,t\in(0,T_{\rm max}).
\eess
Recalling the $v$-equation of \eqref{1.1}, by Lemma \ref{l2.3a}, one can find $C_6>0$ fulfilling
\bess
\|v(\cdot,t)\|_{W^{1,\infty}(\Omega)}\le C_6,\;\; \forall\,t\in(0,T_{\rm max}).
\eess

With the uniform-in-time $W^{1,\infty}(\Omega)$-boundedness of $v$ at hand, the uniform-in-time $L^\infty$ boundedness of $w$ can be proved by using the comparison principle, and then we can perform a Moser-type iterative argument (e.g. [20, Lemma A.1]) to deduce the uniform-in-time $L^\infty$-boundedness of $u$. We therefore finish the proof.
\end{proof}

\noindent{\bf Proof of Theorem \ref{t1.3}.}
The theorem follows directly from the results of Lemma \ref{l2.1} and Lemma \ref{l5.4}.\hfill{\fontsize{8.7pt}{8.7pt}$\Box$}

{\bf Acknowledgment} The work of the first author was supported by National Natural Science Foundation of China 12101519, and the work of the  second author was supported by National Natural Science Foundation of China 12171120.

The authors are very grateful to Professors N. Bellomo and F. Brezzi for their valuable suggestions and provided us with the Ref. \refcite{Bur24}, and thank the anonymous referees for their helpful comments and suggestions.


\begin{thebibliography}{00}

\bibitem{Be20}N. Bellomo, R. Bingham, M. Chaplain, et al., A multi-scale model of virus pandemic: Heterogeneous interactive entities in a globally connected world, {\it Math. Mod. Meth. Appl. Sci.} {\bf 30} (2020) 1591--1651.

\bibitem{Be21}N. Bellomo, F. Brezzi and M. Chaplain, Modelling virus pandemics in a globally connected world, a challenge towards a mathematics for living lystems, {\it Math. Mod. Meth. Appl. Sci.} {\bf 31} (2021)  2391--2397.

\bibitem{Be24}N. Bellomo, R. Eftimie and G. Forni, What is the in-host dynamics of the SARS-CoV-2 virus? A challenge within a multiscale vision of living systems, {\it Netw. Heterog. Media} {\bf 19} (2024) 655--681.

\bibitem{BOS22} N. Bellomo, N. Outada, J. Soler, et al., Chemotaxis and cross-diffusion models in complex environments: Models and analytic problems toward a multiscale vision, {\it Math. Mod. Meth. Appl. Sci.} {\bf 32}(4) (2022) 713--792.

\bibitem{BPTM19}N. Bellomo, K. J. Painter, Y. Tao and M. Winkler, Occurrence vs. absence of taxis-driven instabilities in a May-Nowak model for virus infection, {\it SIAM J. Appl. Math.} {\bf 79} (2019) 1990--2010.

\bibitem{BeT20}N. Bellomo and Y. Tao, Stabilization in a chemotaxis model for virus infection, {\it Discrete Contin. Dyn. Syst. Ser. S} {\bf 13} (2020) 105--117.

\bibitem{Bo97}S. Bonhoeffer, R. M. May, G. M. Shaw and M. A. Nowak, Virus dynamics and drug therapy, {\it Proc. Natl. Acad. Sci. USA} {\bf 94} (1997) 6971--6976.

\bibitem{Bur24}D. Burini and D.A. Knopoff, Epidemics and society-A multiscale vision from the small world to the globally interconnected world, {\it Math. Mod. Meth. Appl. Sci}. {\bf 34}(9) (2024) 1567--1596.

\bibitem{CT21RWA}X. Cao and Y. Tao, Boundedness and stabilization enforced by mild saturation of taxis in a producer scrounger model, {\it Nonlinear Anal.: RWA} {\bf 57}  (2021) 103189.

\bibitem{C-Sjde12}T. Cie\'{s}lak and C. Stinner, Finite-time blowup and global-in-time unbounded solutions to a parabolic-parabolic quasilinear Keller-Segel system in higher dimensions, {\it J. Differential Equations} {\bf 252}(10) (2012) 5832--5851.

\bibitem{C-Sjde15}T. Cie\'{s}lak and C. Stinner, New critical exponents in a fully parabolic quasilinear Keller-Segel system and
applications to volume filling models, {\it J. Differential Equations} {\bf 258}(6) (2015) 2080--2113.

\bibitem{FLM10}M. Di Francesco, A. Lorz and P. Markowich, Chemotaxis-fluid coupled model for swimming bacteria with nonlinear diffusion: global existence and asymptotic behavior, {\it Discrete Contin. Dyn. Syst}. {\bf 28} (2010) 1437--1453.

\bibitem{Fuest19}M. Fuest, Boundedness enforced by mildly saturated conversion in a chemotaxis-May-Nowak model for virus infection, {\it J. Math. Anal. Appl}. {\bf 472}  (2019) 1729--1740.

\bibitem{FIWY16}K. Fujie, A. Ito, M. Winkler and T. Yokota, Stabilization in a chemotaxis model for tumor invasion, {\it Discrete Contin. Dyn. Syst}. {\bf 36} (2016) 151--169.

\bibitem{Het00}H. W. Hethcote, The mathematics of infectious diseases, {\it SIAM Rev}. {\bf 42} (2000) 599--653.

\bibitem{ISY14}S. Ishida, K. Seki and T. Yokota, Boundedness in quasilinear Keller-Segel systems of parabolic-parabolic type on non-convex bounded domains, {\it J. Differential Equations} {\bf 256} (2014) 2993--3010.

\bibitem{Ke70}E. Keller E and L. Segel, Initiation of slime mold aggregation viewed as an instability, {\it J. Theoret. Biol}. {\bf 26} (1970) 399--415.

\bibitem{Ko04} A. Korobeinikov, Global properties of basic virus dynamics models, {\it Bull. Math. Biol}. {\bf 66} (2004) 879--883.

 \bibitem{Lan20}J. Lankeit, Infinite time blow-up of many solutions to a general quasilinear parabolic-elliptic Keller-Segel system, {\it Discrete Contin. Dyn. Syst. Ser. S} {\bf 13} (2020) 233--255.

\bibitem{L-M17}P. Lauren\c{c}ot and N. Mizoguchi, Finite time blowup for the parabolic-parabolic Keller-Segel system with critical diffusion, {\it Ann. Inst. Henri Poincar\'{e} C} {\bf 34}(1) (2017) 197--220.

\bibitem{LKOH17} S. Lee, S. Kim, Y. Oh and H. J. Hwang, Mathematical modeling and its analysis for instability of the immune system induced by chemotaxis, {\it J. Math. Biol}. {\bf 75} (2017) 1101--1131.

\bibitem{LZ24}Y. Li and Q. Zhang, Blow-up prevention by logistic damping in a chemotaxis-May-Nowak model for virus infection, {\it Results Math}. {\bf 79} (2024)  152.

\bibitem{Nowak-B96} M. A. Nowak and C. R. M. Bangham, Population dynamics of immune responses to persistent viruses, {\it Science} {\bf 272} (1996) 74--79.

\bibitem{Nowak2000}M. A. Nowak and R. May, {\it Virus dynamics: Mathematical principles of immunology and virology} (Oxford University Press, Oxford, UK, 2000).

\bibitem{XuWH-ZAMP2021}X. Pan, L. Wang and X. Hu, Boundedness and stabilization of solutions to a chemotaxis May-Nowak model, {\it Z. Angew. Math. Phys}. {\bf 72}  (2021) 52.

\bibitem{Stancevic}O. Stancevic, C. Angstmann, J.M. Murray and B.I. Henry, Turing Patterns from Dynamics of Early HIV Infection, {\it Bull. Math. Biol}. {\bf 75} (2013) 774--795.

\bibitem{TW-S11}Y. Tao and M. Winkler, A Chemotaxis-Haptotaxis Model: The Roles of Nonlinear Diffusion and Logistic Source, {\it SIAM J. Math. Anal}. {\bf 43} (2011) 685--704.

\bibitem{Tao-Wjde12}Y. Tao and M. Winkler, Boundedness in a quasilinear parabolic-parabolic Keller-Segel system with subcritical sensitivity, {\it J. Differential Equations} {\bf 252}(1) (2012) 692--715.

\bibitem{TW-S21}Y. Tao and M. Winkler, Taxis-driven formation of singular hotspots in a may-nowak type model for virus infection, {\it SIAM J. Appl. Math}. {\bf 53} (2021) 1411--1433.

\bibitem{Tao-Z23}X. Tao and S. Zhou, Boundedness in a chemotaxis-May-Nowak model for virus dynamics with mildly saturated chemotactic sensitivity and conversion, {\it Discrete Contin. Dyn. Syst. Ser. B} {\bf 28} (2023) 5269--5280.

\bibitem{WW-M3AS2020}J. Wang and M. Wang, Global bounded solution of the higher-dimensional forager-exploiter model with/without growth sources, {\it Math. Mod. Meth. Appl. Sci}. {\bf 30}(7) (2020) 1297--1323.

\bibitem{WangS23} J. Wang and X. Si, Global solutions to a chemotaxis-May-Nowak model with arbitrary superlinear degradation, {\it Discrete Contin. Dyn. Syst. Ser. B} {\bf 28} (2023) 5281--5295.

\bibitem{Wang-W23} Y. Wang and M. Winkler, Finite-time blow-up in a repulsive chemotaxis-consumption system, {\it Proc. Royal Soc. Edinb}. {\bf A153} (2023)  1150--1166.

\bibitem{WZM20} R. Willie, P. Zheng, N. Parumasur and C. Mu, Asymptotic and stability dynamics of an HIV-1-cytotoxic T lymphocytes (CTL) chemotaxis model, {\it J. Nonlinear Sci}. {\bf 30} (2020) 1055--1080.

\bibitem{Winkler10} M. Winkler, Does a `volume-filling effect' always prevent chemotactic collapse, {\it Math. Meth. Appl. Sci}. {\bf 33}(1) (2010), 12.

\bibitem{Win-jde10} M. Winkler, Aggregation vs. global diffusive behavior in the higher-dimensional Keller-Segel model, {\it J. Differential Equations} {\bf 248} (2010) 2889--2905.

\bibitem{WinD-10}M. Winkler, K. C. Djie, Boundedness and finite-time collapse in a chemotaxis system with volume-filling effect, {\it Nonlinear Analysis} {\bf 72} (2010) 1044--1064.

\bibitem{Win-19} M. Winkler, Boundedness in a chemotaxis-May-Nowak model for virus dynamics with mildly saturated chemotactic sensitivity, {\it Acta Appl. Math}. {\bf 163} (2019) 1--17.

\bibitem{Winkler-JDE2019} M. Winkler, Global classical solvability and generic infinite-time blow-up in quasilinear Keller-Segel systems with
bounded sensitivities, {\it J. Differential Equations} {\bf 266} (2019) 8034--8066.

\bibitem{Winkler-JDE2018} M. Winkler, Global existence and stabilization in a degenerate chemotaxis-Stokes system with mildly strong diffusion, {\it J. Differential Equations} {\bf 264} (2018) 6109--6151.

\bibitem{Winkler-PLMS2022} M. Winkler, A family of mass-critical Keller-Segel systems, {\it Proc. Lond. Math. Soc}. {\bf 124}(2) (2022) 133--181.

\bibitem{YKH20} C. Yoon, S. Kim and H. J. Hwang, Global well-posedness and pattern formations of the immune system induced by chemotaxis, {\it Math. Biosci. Eng}. {\bf 17} (2020) 3426--3449.

\bibitem{ZSL23} P. Zheng, W. Shan and G. Liao, Stability analysis of the immune system induced by chemotaxis, {\it SIAM J. Appl. Dyn. Syst}. {\bf 22} (2023) 2527--2569.
\end{thebibliography}
\end{document}